\documentclass{amsart}

\usepackage{amsmath}
\usepackage{amsfonts}
\usepackage{amssymb}
\usepackage{hyperref}
\usepackage{subfig}
\usepackage{graphicx}
\usepackage{array}
\usepackage{indentfirst}
\usepackage{cite}
\usepackage{esint}
\usepackage{enumerate}

\newtheorem{theorem}{Theorem}[section]
\newtheorem{lemma}[theorem]{Lemma}
\newtheorem{corollary}[theorem]{Corollary}

\theoremstyle{definition}
\newtheorem{definition}[theorem]{Definition}

\theoremstyle{remark}
\newtheorem{remark}[theorem]{Remark}

\numberwithin{equation}{section}

\begin{document}

\title{Defects of liquid crystals with variable degree of orientation}

\author{Onur Alper}
\address[Onur Alper]{Courant Institute of Mathematical Sciences, \newline \indent 251 Mercer Street, New York, NY 10012, USA}
\email{alper@cims.nyu.edu}

\author{Robert Hardt}
\address[Robert Hardt]{Mathematics Department, Rice University, \newline \indent Houston, TX 77251, USA}
\email{hardt@math.rice.edu}

\author{Fang-Hua Lin}
\address[Fang-Hua Lin]{Courant Institute of Mathematical Sciences, \newline \indent 251 Mercer Street, New York, NY 10012, USA}
\email{linf@cims.nyu.edu}





\begin{abstract}
The defect set of minimizers of the modified Ericksen energy for nematic liquid crystals
consists locally of a finite union of isolated points and H\"older continuous curves with finitely many crossings.
\end{abstract}

\maketitle

\section{Introduction}

\subsection{The Ericksen model and related works}

A liquid crystal is a liquid exhibiting some anisotropic optical behavior. 
For each point in a spatial region $\Omega \subset \mathbf{R}^3$, consider a probability distribution $l$ of unit vectors for the direction of a symmetric elongated molecule.
While the first moment (i.e. average) of $l$ is $0$ by symmetry, the second moment $\left \langle l \otimes l \right \rangle = \left \langle l_i l_j \right \rangle$ can
reveal the anisotropy. Note that the second moment is given by a positive, symmetric matrix with trace $1$. 
In the uniaxial nematic regime, we assume that its traceless part has $2$ equal eigenvalues. Hence, in this regime, we can write:
\begin{equation*}
\left \langle l \otimes l \right \rangle - \frac{1}{3} id = s \left [ (n \otimes n ) - \frac{1}{3} id \right ],
\end{equation*}
where $| n | = 1$, $s = \frac{3}{2} \left \langle l \otimes n \right \rangle^2 - \frac{1}{2} \in \left [ -1/2, 1 \right ]$.
The director field is completely ordered $(l = \pm n )$ in the case $s=1$, completely random $\left ( \langle l \otimes n \rangle^2 = 1/3 \right )$ for $s=0$,
and completely orthogonal $( l \cdot n \equiv 0 )$ for $s=-1/2$.

For $s$ constant, the assumptions of at most quadratic dependence on the gradient, frame indifference and material symmetry lead to 
the Oseen-Frank free energy $\int_{\Omega} W(n(x)) \, \mathrm{d}x$ for:
\begin{equation*}
 W(n) = \kappa_1 | \mathrm{div} n |^2 + \kappa_2 | n \cdot \mathrm{curl} n |^2 + \kappa_3 | n \times \mathrm{curl} n |^2 
+ \left ( \kappa_2 + \kappa_4 \right ) \left [ \mathrm{tr} ( \nabla n )^2 - ( \mathrm{div}n )^2 \right ],
\end{equation*}
where $\kappa_1$, $\kappa_2$, $\kappa_3 > 0$. Note that the special case $\kappa_1 = \kappa_2 = \kappa_3 = 1$ and $\kappa_4 = 0$ corresponds to $W(n) = | \nabla n |^2$,
and hence critical unit vector fields coincide with harmonic maps from $\Omega$ to $\mathbf{S}^2$. By \cite{SU82} minimizers have only isolated singularities in this special case.
For the general Oseen-Frank energy, it is proved in \cite{HKL86} that minimizers exist and the set of their discontinuities always has Hausdorff dimension strictly less than $1$.
In particular, the Oseen-Frank theory does not allow finite energy minimizing configurations with a line singularity. 

In 1985 J.L. Ericksen suggested a model with variable $s$.
As in \cite{Maddocks} and \cite{Ericksen91}, the energy $\int_{\Omega} X(s,n) \, \mathrm{d}x$ is proposed, where:
\begin{equation*}
X(s,n) = s^2 W(n) + \kappa_5 | \nabla s|^2 + \kappa_6 | \nabla s \cdot n |^2 + \psi(s),
\end{equation*}
with a $C^2$-potential $\psi$ satisfying the following requirements: 
\begin{enumerate}[(i)]
 \item $\lim_{s \to -1/2} \psi(s) = + \infty$,
 \item $\lim_{s \to 1} \psi(s) = + \infty$,
 \item $\psi'(0) = 0$,
 \item $\psi$ has a minimum at some $s_* \in (0,1)$.  
\end{enumerate}
See \cite[Fig.1]{Maddocks}.
A challenge that the integrand $X(s,n)$ poses when $W(n)$ depends quadratically on $\nabla n$ is the choice of a function space for $n$. Due to the factor $s^2$ in front of $W(n)$,
it is not natural to require $\nabla n$ to be square-integrable. Furthermore, the ellipticity of formal Euler-Lagrange equations degenerates on the zero set of $s$.

In \cite{Lin89} and \cite{Lin91}, in the case $\kappa_1 = \kappa_2 = \kappa_3 = 1$, $\kappa_4 = \kappa_6 = 0$, and $\kappa_5 = k$, the existence and regularity theory for a minimizing pair
$(s,n)$ is related to that of a minimizing harmonic map into a cone, via recasting them as: $u = \left ( |k-1|^{1/2} s, sn \right )$ and observing that:
\begin{equation*}
X(s,n) = k |\nabla s |^2 + s^2 | \nabla n |^2 + \psi(s) = | \nabla u |^2 + \psi \left ( k^{-1/2} |u| \right ), 
\end{equation*}
where the image of $u$ is constrained to lie in the round cone:
\begin{equation*}
\mathbf{C}_k = \left \{ (y,z) \in \mathbf{R}^3 \times \mathbf{R} \, : \, z = |k-1|^{1/2} |y| \right \}. 
\end{equation*}
One can also consider the corresponding two-sheeted cone, hence allowing $s$ to take negative values as well, but here we restrict our attention the case $s \geq 0$.  
For $k>1$, $\mathbf{C}_k$ is positively curved (in the Alexandrov sense) with the metric induced from Euclidean space $\mathbf{R}^3 \times \mathbf{R}$, 
while for $k=1$ it is flat in Euclidean space $\mathbf{R}^3 \times \{0\}$. 
Finally for $k \in (0,1)$ it is negatively curved with the metric induced from Minkowski space $\mathbf{R}^{3+1}$. Furthermore, in the last case, the ambient metric is positive definite
when restricted to $\mathbf{C}_k$. 
In particular, it is proved in \cite{Lin91} that $u$ is locally H\"older continuous in the case $k>1$, locally Lipschitz in the case $k \in (0,1)$ and smooth in the case $k=1$. 
Moreover, in all cases, $u$ is analytic away from the preimage of the vertex $u^{-1} \{ 0 \} = s^{-1} \{ 0 \}$, and $\mathrm{sing}(n) = u^{-1} \{0\}$, which we define as the \emph{defect set}.

We note that \cite{Ambrosio90} and \cite{Ambrosio90R} also address the existence and regularity of minimizing pairs $(s,n)$, yet without making use of the harmonic map formulation. 
The question of existence and regularity for the Ericksen model with general material constants was also addressed in \cite{LinPoon94}.
See also \cite{Wang2001} for a generalization of the techniques in \cite{Lin89} and \cite{Lin91} to the context of energy minimizing maps into more general Lipschitz targets. 

In addition, using a dimension reduction argument based on the monotonicity of Almgren frequency, the following Hausdorff dimension estimates were proved in \cite{Lin89}, \cite{Lin91}
for nontrivial minimizing maps:
$\mathrm{dim} \left ( u^{-1} \{0\} \right ) \leq 2$ for $k \in (0,1]$ and  $\mathrm{dim} \left ( u^{-1} \{0\} \right ) \leq 1$ for $k >1$.
By \cite{AmbrosioVirga91}, the first estimate is sharp, as \emph{wall defects} occur in this case. 
On the other hand, the second estimate was improved in \cite{HardtLin93}. By proving that there are no homogeneous, energy minimizing maps depending only on two variables in this case, 
it was shown that there cannot be any such tangent maps either. 
Hence, by the dimension reduction argument in \cite{Lin89}, \cite{Lin91}, the defect set must consist of isolated points in the case $k>1$.

Since an important goal is to understand the experimentally observed $1$ dimensional defects in nematic liquid crystals, a modified Ericksen model was also introduced in \cite{HardtLin93}.
As the head and tail of a nematic liquid crystal molecule are indistinguishable, it is natural to describe the director field using the projective plane
$\mathbf{RP}^2 = \left \{ [y] \, : \, y \in \mathbf{S}^2 \right \}$, where $[y] = \{ y , - y \}$ is the sign equivalence class for $y \in \mathbf{R}^3$, and $\mathbf{RP}^2$ inherits
a round metric from its double cover $\mathbf{S}^2$. Hence, in the context of Ericksen's variable degree of orientation model, replacing $\mathbf{S}^2$ with $\mathbf{RP}^2$, we get
the cone:
\begin{equation*}
\mathbf{D}_k = \left \{ \left ( [y], z \right ) \, : \, (y,z) \in \mathbf{C}_k \right \},
\end{equation*}
which inherits a round metric from $\mathbf{C}_k$. 

The existence, uniqueness and regularity theory for $\mathbf{C}_k$-valued energy minimizing maps carry over veribatim to the case of $\mathbf{D}_k$-valued energy minimizing maps.
However, an important difference is that simple closed geodesics, i.e. great circles of length $\pi$, are not contractible in $\mathbf{RP}^2$, unlike those in $\mathbf{S}^2$. 
While $\mathbf{RP}^2$-valued maps minimizing the Dirichlet or Oseen-Frank energies do not differ from their $\mathbf{S}^2$-valued counterparts 
in terms of the size of their singular sets or their asymptotic behavior near them (cf. \cite{BCL86} and \cite{HKL86}), the nontrivial topology of $\mathbf{RP}^2$ does have an effect
in the context of Ericksen's variable degree of orientation model. In particular, one observes line defects for $\mathbf{D}_k$-valued energy minimizing maps in the case $s \geq 0$, $k>1$.
See Remark \ref{linedefect} for an example. In particular, the Hausdorff dimension estimate $\mathrm{dim} \left ( u^{-1} \{ 0 \} \right ) \leq 1$ is optimal for $\mathbf{D}_k$-valued
energy minimizing maps, in contrast to the case of $\mathbf{C}_k$-valued ones. 

See also \cite{BZ2011} for an extensive discussion on the use of $\mathbf{RP}^2$ in modeling uniaxial nematics, relations with the Landau-de Gennes theory, 
issues of suitable function spaces, orientability of line fields, and boundary conditions.
For another result on line defects in liquid crystals, see \cite{Canevari}, which considers the vanishing elastic constant limit in a Landau-de Gennes model,
in the spirit of the Ginzburg-Landau theory.

\subsection{The main result and an overview}

Our main result is that for any energy minimizing $u \, : \, \Omega \to \mathbf{D}_k$ and any compact $K \subset \Omega$, $u^{-1} \{0\} \cap K$
consists of isolated points and a finite union of H\"older continuous curves with only finitely many crossings.
The finiteness of $\mathcal{H}^1 \left ( u^{-1} \{0\} \cap K \right )$ is proved in \cite{ARect} by the first author
by combining the blow-up analysis in this article with the ideas in \cite{NV} and \cite{DLMSV}.
Together with Corollary \ref{mainresult}, this result implies the local rectifiability of the {\emph{defect set}} $u^{-1} \{0\}$.

As we are interested in the local behavior of the zero set, we can assume $\psi \equiv 0$. 
Our results are also valid in the presence of a $C^2$-potential $\psi$ with $\psi'(0) =0$, as the variational identities and monotonicity formulas in Lemma \ref{monotonocity}  
have their generalizations in \cite[Section 3]{Lin89} and \cite[Section 3]{Lin91}.

The strategy of the proof is parallel to the analysis of the singular set of minimizing harmonic maps from a domain in $\mathbf{R}^4$ to $\mathbf{S}^2$ in \cite{HardtLin90}. 
While the monotonicity of the renormalized Dirichlet energy is the key element in analyzing the singular set in \cite{HardtLin90}, the monotonicity of the Almgren frequency
plays an analogous central role in understanding the zero set. 

As in \cite{HardtLin90}, our main goal is to verify the hypothesis of
Reifenberg's topological disk theorem away from finitely many points in $u^{-1} \{ 0 \}$. 
We also prove that for every $\epsilon >0$,  at all but finitely many points $b \in u^{-1} \{ 0 \}$,
at all scales $r>0$, there is a line $L_{b,r}$, depending on both $b$ and $r$, such that $u^{-1} \{0 \}$ and $L_{b,r}$ both restricted to $B_r(b) \subset \Omega$ have Hausdorff
distance less than $r \epsilon$ between them. If the conclusion holds for all $b$ in a neighborhood, then Reifenberg's theorem states that the defect set must be an embedded
H\"older curve in that neighborhood, cf. \cite{Reifenberg60}.
By showing that there can be at most finitely many points where the hypothesis possibly fails, we infer that the defect set must be a union
of finitely many H\"older curves with finitely many crossings and isolated points.

In Section \ref{Preliminaries}, we introduce the key definitions, monotonicity, compactness and (partial) regularity results from \cite{Lin89}, \cite{Lin91}, which we will use in the rest of the paper. 
In Section \ref{HomogeneousMinimizers}, we consider homogeneous minimizers that arise as the blow-up limits of the minimizing map $u$ rescaled near its zeros. 
In Section \ref{Decomposition}, we decompose the zero set based on the zeros of tangent maps at each of its points. This decomposition allows us to arrive at the following conclusion: 
near the top dimensional part of the zero set, at every scale, $u$ can be approximated by a corresponding homogeneous minimizer, and the remaining part of the zero set is discrete. 
Finally, in Section \ref{Structure}, we prove the main result on the structure of the defect set.

\section{Preliminaries} \label{Preliminaries}

We recall some results from \cite{Lin89} and \cite{Lin91} that will play a key role in the analysis of the defect set.

\begin{lemma} \label{subharmonicity}
For $\Omega \subset \mathbf{R}^d$ and any energy minimizing map $u : \Omega \to \mathbf{D}_k$, the equation
\begin{equation}
\Delta | u |^2 = 2 | \nabla u |^2
\label{subhdistance} 
\end{equation}
holds in the sense of distributions. 
In particular, $\left | u \right |^2$, the squared distance of $u(x)$ from the vertex of $\mathbf{D}_k$, is a subharmonic function on $\Omega$.
Furthermore, $u$ satisfies the following stationarity identity:
\begin{equation}
\mathrm{div} \left ( 2 \nabla u \otimes \nabla u - | \nabla u |^2 id \right ) = 0. 
\label{StationarityIdentity}
\end{equation}
In other words, the stress-energy tensor of $u$ is divergence free.
\end{lemma}
\begin{proof}
Observe that for $\phi \in C_0^\infty(\Omega)$, $u_t = \left ( 1 + t \phi \right ) u$ maps $\Omega$ to $\mathbf{D}_k$, and hence it is an admissible variation.
By the minimality of $u$ we have:
\begin{equation*}
0 = \frac{\mathrm{d}}{\mathrm{d}t}\Bigr |_{t=0} \int_{\Omega} \left | \nabla u_t \right |^2 \, \mathrm{d}x 
= \int_{\Omega} \left \langle \nabla |u|^2, \nabla \phi \right \rangle \, \mathrm{d}x + \int_{\Omega} 2 \left | \nabla u \right |^2  \phi \, \mathrm{d}x,
\end{equation*}
for any $\phi \in C_0^\infty(\Omega)$, which is equivalent to \eqref{subhdistance}.

As $u$ is an energy minimizing map from $\Omega$ into $\mathbf{D}_k$, it is stationary with respect to the domain variations $u \left ( \Psi_t (x) \right )$, where
$\Psi_t \, : \, B_r(a) \to B_r(a)$, $t \in (-\epsilon, \epsilon)$, is a $C^1$-family of diffeomorphisms satisfying $\Psi \in C_0^1 \left ( B_r(a) \right )$ and $\Psi_0(x) = x$.
Therefore, considering the first variation of the Dirichlet energy with respect to a family of diffeomorphisms generated by an arbitrary vector field $X$, 
we obtain the stationarity identity \eqref{StationarityIdentity} in the sense of distributions.
\end{proof}

Before stating the next result, we need to introduce new notations for key quantities that play a role in our analysis. We state these results
in the general case of $d$ dimensions.
\begin{subequations}
\begin{align}
D(a;r) = \int_{B_r(a)} | \nabla u |^2 \, \mathrm{d}x, \label{Dirichlet} \\
E(a;r) = r^{2-d} \int_{B_r(a)} | \nabla u |^2 \, \mathrm{d}x, \label{renormDirichlet} \\
H(a;r) = \int_{ \partial B_r(a)} | u |^2 \, \mathrm{d}A, \label{sdistance} \\
N(a;r) = \frac{r D(a;r)}{H(a;r)}. \label{Almgren}
\end{align}
\end{subequations}
Note that $D(a;r)$ is the Dirichlet energy of the map $u$, $E(a;r)$ is the renormalized Dirichlet energy, and $N(a;r)$
Almgren originally defined in \cite{Almgren83}. Since $|u|^2$ is subharmonic, $H(a;r)$ is a natural quantity to consider as well.

\begin{lemma} \label{monotonocity}
For any energy minimizing map $u \, : \, \Omega \to \mathbf{D}_k$,
for almost every $r \in \left (0, \mathrm{dist}(a,\partial \Omega) \right )$, the following monotonicity identities hold:
\begin{subequations}
\begin{align}
\frac{\mathrm{d}}{\mathrm{d}r} \left [ r^{1-d} H(a;r) \right ] =  2 r^{1-d} D(a;r) \geq 0, \label{subharmonicaverage} \\
\frac{\mathrm{d}}{\mathrm{d}r}  E(a;r)  = \frac{2}{r^{d-2}} \int_{\partial B_r(a)}
\left | \frac{\partial u}{\partial | x - a | } \right |^2 \, \mathrm{d}A \geq 0, \label{harmonicmapmonotonicity} \\
\frac{\mathrm{d}}{\mathrm{d}r} N(a;r) \geq 0; \label{almgrenmonotonicity}
\end{align}
\end{subequations}
the limit $N \left ( a; 0^+ \right ) = \lim_{r \downarrow 0} N(a;r) $ exists, and $N \left ( a; 0^+ \right )$ is upper semicontinuous in $a$.
Moreover, \eqref{almgrenmonotonicity} holds with equality, if and only if $u$ is homogeneous of degree $N \left ( a; 0^+ \right )$ with respect to $|x-a|$.
\end{lemma}

\begin{proof}
The proofs are in \cite{Lin89} and \cite{Lin91}. We sketch them for completeness.
Note:
\begin{equation}
\begin{split}
H'(a;r) & = \frac{\mathrm{d}}{\mathrm{d}r} \left ( r^{d-1} \int_{\partial B_1(0)} |u|^2 (a+rx) \, \mathrm{d}A \right ) \\
	& = (d-1)r^{d-2} \int_{\partial B_1(0)} |u|^2 (a+rx) \, \mathrm{d}A + \int_{\partial B_r(a)} \partial_{|x-a|} \left ( |u|^2 \right ) \mathrm{d}A \\
	& = \frac{d-1}{r} H(a;r) + 2 D(a;r).
\end{split} \label{hprime}
\end{equation}
The last equality is obtained by testing \eqref{subhdistance} with $\phi_k \in C_0^\infty \left ( B_r(a) \right )$ approximating the characteristic function of $B_r(a)$ 
and the divergence theorem:
\begin{equation}
\begin{aligned}
 D(a;r) = \lim_{k \to \infty} \frac{1}{2} \int_{B_r(a)} \left < \nabla |u|^2, \nabla \phi_k \right > \mathrm{d}x & =
\lim_{k \to \infty} \frac{1}{2} \int_{\partial B_r(a)} \partial_{|x-a|} \left ( |u|^2 \right ) \phi_k \, \mathrm{d}A \\ 
& = \lim_{k \to \infty} \int_{\partial B_r(a) } u \cdot u_{|x-a|} \phi_k \, \mathrm{d}A \label{dividentity} \\
& = \int_{\partial B_r(a) } u \cdot u_{|x-a|} \, \mathrm{d}A.
\end{aligned}
\end{equation}
Now \eqref{subharmonicaverage} follows immediately from \eqref{hprime}.
Likewise, testing \eqref{StationarityIdentity} with vector fields $\phi_k \left ( |x-a| \right ) r^{-1} (x-a)$, 
where $\phi_k \in C_0^\infty \left ( B_r(a) \right )$ approximate the characteristic function
of $B_r(a)$, we obtain \eqref{harmonicmapmonotonicity} in the limit $k \to \infty$.

In order to prove \eqref{almgrenmonotonicity}, namely the monotonicity of Almgren frequency, we take the logarithmic derivative of $N(a;r)$ and get:
\begin{equation}
 \begin{split}
  N'(a;r) & = N(a;r) \frac{\mathrm{d}}{\mathrm{d}r} \left [ \log \left ( r^{2-d} D(a;r) \right ) - \log \left ( r^{1-d}  H(a;r) \right )  \right ] \\
          & =  2 N(a;r) \left [ \frac{\int_{\partial B_r(a)} \left | u_{|x-a|} \right |^2 \, \mathrm{d}A }{D(a;r)} - \frac{D(a;r)}{H(a;r)} \right ], \label{logdiff}
 \end{split}
\end{equation}
where the second equality is obtained by applying \eqref{harmonicmapmonotonicity} and \eqref{hprime} to the first and third terms respectively.
Plugging \eqref{dividentity} in the denominator of the first term and the numerator of the second term in \eqref{logdiff},
and applying the Cauchy-Schwarz inequality to each, we obtain \eqref{almgrenmonotonicity}. The existence of the limit
$N \left ( a; 0^+ \right ) = \lim_{r \downarrow 0} N(a;r) $ is clear from \eqref{almgrenmonotonicity}.

To show the upper semicontinuity of $N \left ( a; 0^+ \right )$, we assume $\lim_{j \to \infty} a_j = a_\infty$ in $\Omega$ and note by \eqref{almgrenmonotonicity}:
\begin{equation*}
\lim \sup_{j \to \infty} N \left ( a_j ; 0^+ \right ) \leq \lim \sup_{j \to \infty} \rho \frac{D \left (a_j ; \rho \right ) }{ H \left ( a_j ; \rho \right ) }
\leq \rho \frac{\lim \sup_{j \to \infty}  D \left (a_j ; \rho \right )}{\lim \inf_{j \to \infty}  H \left ( a_j ; \rho \right )}.
\end{equation*}
Estimating the numerator by observing that $B_{\rho} \left ( a_j \right ) \subset B_{\rho + \epsilon} \left ( a_\infty \right ) $ for large enough $j$,
and applying Fatou's lemma to the denominator, we get:
\begin{equation*}
\lim \sup_{j \to \infty} N \left ( a_j ; 0^+ \right ) \leq \rho \frac{D \left (a_\infty ; \rho + \epsilon \right )}{H \left ( a_\infty; \rho \right )},
\end{equation*}
for every $\epsilon > 0$. First letting $\epsilon \downarrow 0$ and then letting $\rho \downarrow 0$, we obtain:
\begin{equation}
\lim \sup_{j \to \infty } N \left ( a_j ; 0^+ \right ) \leq N \left ( a_\infty; 0^+ \right ).
\label{uppersemicontinuity} 
\end{equation}

Finally, note that \eqref{almgrenmonotonicity} holds with equality, if and only if the Cauchy-Schwarz inequality applied above holds with equality.
But this is equivalent to $u$ being homogeneous with respect to $r = |x-a|$ at almost every $r \in \left (0, \mathrm{dist}(a,\partial \Omega) \right )$.
In this case, we can apply \eqref{dividentity} and the Cauchy-Schwarz inequality to the numerator of $N(a;r)$ to observe that 
the degree of homogeneity has to be $N(a;r)$.
\end{proof}

Next, we state two useful estimates that are closely related to the above monotonicity formulas.

\begin{lemma} \label{doubling}
For any energy minimizing map $u \, : \, \Omega \to \mathbf{D}_k$, $0 < r < R$ and $B_R(a) \subset \Omega$, the following inequalities hold:
\begin{equation}
\left ( \frac{r}{R} \right )^{d-1+2N(a;R)} H(a;R) \leq H(a;r) \leq \left (\frac{r}{R} \right )^{d-1+2N(a;r)} H(a;R). \label{Hcomparison}
\end{equation}
Moreover, whenever $|a| < r < R$ and $|a| + r < R$, for every $\rho \in (0, r )$ we have:
\begin{equation}
N(a;\rho) \leq C_1\left (r, R, |a| \right ) N(0;R) + C_2 \left ( r, R, |a| \right ) . \label{localFcontrol}
\end{equation}
\end{lemma}

\begin{proof}
Using \eqref{subharmonicaverage} we compute:
\begin{equation}
\frac{\mathrm{d}}{\mathrm{d}r} \left [ \log \frac{H(a;r)}{r^{d-1}} \right ] = \frac{r^{d-1}}{H(a;r)} \left [ \frac{2D(a;r)}{r^{d-1}} \right ]
= \frac{2 N(a;r)}{r}.
\label{logH} 
\end{equation}
Integrating \eqref{logH} and using the monotonicity formula \eqref{almgrenmonotonicity} we obtain:
\begin{equation*}
\frac{H(a;R) \cdot r^{d-1}}{H(a;r) \cdot R^{d-1}} = \exp \left ( \int_r^R \frac{2 N(a;s)}{s} \, \mathrm{d}s \right ),
\end{equation*}
from which \eqref{Hcomparison} follows via the frequency monotonicity \eqref{almgrenmonotonicity}.

For the second formula we pick $R' \in (r,R)$ and $r' \in (0,r)$ such that
$B_{R'}(a) \subset B_R(0)$ and $B_{r'}(0) \subset B_r(a)$. Therefore:
\begin{equation*}
\fint_{B_{R'}(a)} |u|^2 \, \mathrm{d}x \leq  \left ( R' \right )^{-d} \int_{B_R(0)} |u|^2 \, \mathrm{d}x = \left ( R' \right )^{-d} \int_0^R H(0;s) \, \mathrm{d}s.
\end{equation*}
Then using \eqref{Hcomparison} and \eqref{almgrenmonotonicity}, we observe:
\begin{equation}
\begin{split}
\fint_{B_{R'}(a)} |u|^2 \, \mathrm{d}x
& \leq \left ( R' \right )^{-d} \int_0^R H \left ( 0; \frac{r'}{R}s \right ) \left ( \frac{R}{r'} \right )^{d-1+2N(0;R)} \, \mathrm{d}s \\
& = \left ( R' \right )^{-d} \left ( \frac{R}{r'} \right )^{d+2N(0;R)} \int_{B_{r'}(0)} |u|^2 \, \mathrm{d}x \\
& \leq \left ( \frac{r}{R'} \right )^d \left ( \frac{R}{r'} \right )^{d+2N(0;R)} \fint_{B_r(a)} |u|^2 \, \mathrm{d}x.
\end{split}
\label{localFcontrolFact1}
\end{equation}
Also note that picking $R'' \in \left ( r, R' \right )$:
\begin{equation}
\int_{B_{R'}(a)} |u|^2 \geq \int_{R''}^{R'} s^{d-1} \fint_{\partial B_s(a)} |u|^2 \, \mathrm{d}A \, \mathrm{d}s 
\geq \frac{\left ( R' \right )^d - \left ( R^{''} \right )^d}{d} \fint_{\partial B_{R^{''}}(a)} |u|^2 \, \mathrm{d}A.
\label{localFcontrolFact2} 
\end{equation}
Likewise:
\begin{equation}
\int_{B_r(a)} |u|^2 \, \mathrm{d}x = \int_0^r s^{d-1} \fint_{\partial B_s(a)} |u|^2 \, \mathrm{d}A \, \mathrm{d}s 
\leq \frac{r^d}{d} \fint_{\partial B_r(a)} |u|^2 \, \mathrm{d}A.
\label{localFcontrolFact3}
\end{equation}
Finally, integrating \eqref{logH} gives:
\begin{equation}
\log \left ( \frac{ \fint_{\partial B_{R^{''}}(a)} |u|^2 \, \mathrm{d}A }{\fint_{\partial B_r(a)} |u|^2 \, \mathrm{d}A} \right ) 
\geq 2N(a;r) \log \left ( \frac{R^{''}}{r} \right ). 
\label{localFcontrolFact4}
\end{equation}
From \eqref{localFcontrolFact1}, \eqref{localFcontrolFact2}, \eqref{localFcontrolFact3} and \eqref{localFcontrolFact4}, we conclude:
\begin{equation}
N(a; \rho) \leq N(a;r) \leq \frac{\log \left (\frac{\left ( Rr / r' \right )^d}{\left (R' \right )^d - \left ( R^{''} \right )^d} \right )}{2 \log \left ( R^{''}/r \right )}
+ \frac{\log \left ( R / r' \right )}{ \log \left ( R^{''} / r \right )} N(0;R),
\label{exactLocalFcontrol}
\end{equation}
for every $\rho \in (0, r)$.
Now we can fix $r'$, $R'$ and $R''$ with respect to $|a|$, $r$ and $R$ in terms of suitable multiplicative factors and get the constants $C_1$ and $C_2$
as in \eqref{localFcontrol}.
\end{proof}

Note that $\mathbf{D}_k$ is not compact, whereas the physically relevant range for $s$ is $\left [ - \frac{1}{2}, 1 \right ]$.
We cannot drop the assumption that $s$ is nonnegative, as relaxing this assumption leads to the so-called wall defects, cf. \cite{AmbrosioVirga91}.
However, the following maximum principle allows us to guarantee $| u | \leq 1$, and hence $s \leq 1$, as long as the boundary data satisfies
the same bound.

\begin{lemma} \label{maximumprinciple}
For any $u \, : \, \Omega \to \mathbf{D}_k$ energy minimizing with respect to the boundary data $u_0 \in H^{\frac{1}{2}}(\partial \Omega)$
satisfying $\left | u_0 \right | \leq M$ for a constant $M \geq 1$,  we have $| u | \leq M$.  
\end{lemma}
\begin{proof}
We follow \cite[Lemma 4.10.2]{Jost84} closely.
For $\mathbf{D}_{k,M} = \left \{ y \in \mathbf{D}_k \, : \, |y| \leq M \right \}$, we can define the projection $\Pi_{M}: \mathbf{D}_k \to \mathbf{D}_{k,M}$,
as follows: since each $y \in \mathbf{D}_k \backslash \mathbf{D}_{k,M}$ is connected to the vertex $0$ with a unique geodesic line in the cone $\mathbf{D}_k$,
which intersects $\partial \mathbf{D}_{k,M}$ normally, we can define $\Pi_M(y)$ as this point of intersection. 
Then $\Pi_M$ is a Lipschitz retraction of $\mathbf{D}_k$ to $\mathbf{D}_{k,M}$. 
It is distance-decreasing in $\mathbf{D}_k \backslash \mathbf{D}_{k,M}$. 
Therefore, it satisfies $\left | \mathrm{d}\Pi_M (\xi) \right | < |\xi|$ for $\xi \in \mathrm{T}_x\mathbf{D}_k$, $x \in \mathbf{D}_k \backslash \mathbf{D}_{k,M}$. 
Moreover, $\Pi_M$ is equal to identity on $\mathbf{D}_{k,M}$.

Hence, $\Pi_M \left ( u_0 \right ) = u_0$, while:
\begin{equation*}
\begin{split}
\int_{\Omega} \left | \nabla \left ( \Pi_M \circ u \right ) \right |^ 2 \, \mathrm{d}x & 
=  \int_{\left \{ | u| > M \right \}} \left | \nabla \left ( \Pi_M \circ u \right ) \right |^ 2 \, \mathrm{d}x
+ \int_{\left \{ | u| \leq M \right \}} \left | \nabla \left ( \Pi_M \circ u \right ) \right |^ 2 \, \mathrm{d}x \\
& = \int_{\Omega \cap \left \{ | u| > M \right \}} \left | \mathrm{d} \Pi_M \left ( \nabla u \right ) \right |^ 2 \, \mathrm{d}x +
\int_{\Omega \cap \left \{ | u| \leq M \right \}} \left | \nabla u \right |^ 2 \, \mathrm{d}x \\
& < \int_{\Omega \cap \left \{ | u| > M \right \}} \left |  \nabla u \right |^ 2 \, \mathrm{d}x
+ \int_{\Omega \cap \left \{ | u| \leq M \right \}} \left | \nabla u \right |^ 2 \, \mathrm{d}x = \int_{\Omega} \left | \nabla u \right |^ 2 \, \mathrm{d}x, 
\end{split}
\end{equation*}
contradicting the minimality of $u$, unless $|u| \leq M$. 
\end{proof}

We state two useful consequences of $\mathbf{D}_k$ being a simply connected Lipschitz target. The proofs follow closely from \cite{HardtLin87} and \cite{HKL88}, but we include them
for completeness.

\begin{lemma} \label{InteriorEnergyEstimate}
For any compact subset $K$ of $\Omega$, there exists a constant
$D_0 = D_0 \left ( \| u \|_{L^{\infty}(\partial \Omega)}, K, \Omega \right )$ such that
\begin{equation}
\int_{K} | \nabla u |^2 \, \mathrm{d}x \leq D_0,
\label{UniformIntEnergy}
\end{equation}
for any energy-minimizing $u \, : \, \Omega \to \mathbf{D}_k$. 
\end{lemma}

\begin{proof}
By the compactness of $K$, scaling and translation, it suffices to prove the estimate for $\Omega = B_1(0)$ and $K = \overline{B_{1-\epsilon}(0)}$ for a fixed $\delta \in (0,1)$.

By Lemma \ref{maximumprinciple}, the image of $u$ is contained in $B_R(0) \cap \mathbf{D}_k$ for some $R > 0$.
Although $B_R(0) \cap \mathbf{D}_k$ is not a Lipschitz submanifold near the origin, we may apply the argument of \cite[Lemma 6.1]{HardtLin87}.
There is an (easily constructed) bilipschitz triangulation of the closed cube $Q$ in $\mathbf{R}^5$ so that the image $D$ of $\overline{B_R(0)} \cap \mathbf{D}_k$ is a (3 dimensional)
subcomplex. As in the proof of \cite[Lemma 6.1]{HardtLin87}, the fact that $D$ (being topologically a truncated cone) is simply-connected, gives us a {\it one} dimensional Lipschitz
polygon $X \subset Q$ and a suitable locally Lipschitz retraction $P$ of $Q \setminus X$ onto $D$.
This projection map allows us to construct comparison maps to derive estimates for the energy-minimizing map $u$ 
from corresponding estimates on the harmonic extension of $u|_{\partial B_s(a)}$ to $B_s(a)$.

Arguing as in the proof of \cite[Theorem 6.2]{HardtLin87}, while replacing the harmonic estimate in \cite[Lemma 2.1]{HardtLin87} with the one corresponding 
to the estimate below, we conclude that there exists $\omega \in H^1 \left ( B_s(a), \mathbf{D}_k \right )$, an extension of $u|_{\partial B_s(a)}$ to $B_s(a)$ satisfying:
\begin{equation} \label{extensionEstimate1}
\int_{B_s(a)} | \nabla \omega |^2 \, \mathrm{d}x \leq C \left ( \int_{\partial B_s(a)} \left | \nabla_{\mathrm{tan}} u \right |^2 \, \mathrm{d}A 
\cdot \int_{\partial B_s(a)}  |u - \xi |^2 \, \mathrm{d}A \right )^{1/2},  
\end{equation}
for any $\xi \in \mathbf{R}^m$, $B_s(a) \subset \Omega$ and $C = C \left ( \| u \|_{L^{\infty}(\partial \Omega)}, k \right )$, a constant. 

Also note that $D(a;\rho)$ is monotone increasing with
\begin{equation*}
\frac{\mathrm{d}}{\mathrm{d}\rho} D(0;\rho) = \int_{\partial B_{\rho}(0)} | \nabla u |^2 \, \mathrm{d}A,
\end{equation*}
for almost every $\rho \in (0,1)$. Hence, using \eqref{extensionEstimate1} with $\xi = 0$, by the minimality of $u$, we obtain:
\begin{equation*}
\begin{split}
D(0;\rho) \leq \int_{B_\rho(0)} | \nabla \omega |^2 \, \mathrm{d}x 
& \leq C \left (  \frac{\mathrm{d}}{\mathrm{d}\rho} D(0;\rho)
\cdot \int_{\partial B_s(a)}  |u|^2 \, \mathrm{d}A \right )^{1/2} \\
& \leq CM \rho \left [ \frac{\mathrm{d}}{\mathrm{d}\rho} D(0;\rho) \right ]^{1/2},
\end{split}
\end{equation*}
where $M = M \left ( \| u \|_{L^{\infty}(\partial \Omega)} \right )$, and the last inequality is due to Lemma \ref{maximumprinciple} and the H\"older inequality.
Rearranging this estimate as
\begin{equation*}
\frac{(CM)^{-2}}{\rho^2} \leq \frac{\frac{\mathrm{d}}{\mathrm{d}\rho} D(0;\rho)}{D(0;\rho)^2},
\end{equation*}
integrating this inequality from $1-\delta$ to $1$, and dropping the positive term $D(0,1)$, we obtain:
\begin{equation*}
D(a;1-\delta) \leq (CM)^2 \frac{1-\delta}{\delta}.
\end{equation*}
\end{proof}

\begin{lemma} \label{caccioppolitype}
For any energy minimizing map $u \, : \, B_R(0) \to \mathbf{D}_k$, with $\| u \|_{L^\infty} < \infty$,  the following Caccioppoli-type inequality holds for every $\lambda \in (0,1)$:
\begin{equation}
E(a; r) \leq \lambda E(a; 2r) + \frac{C}{\lambda^2} W(a; 2r),
\label{caccioppoli} 
\end{equation}
where $W(a; \rho) = \fint_{B_{\rho}(a)} \left | u - \bar{u} \right |^2 \, \mathrm{d}x$, 
$\bar{u} = \fint_{B_{\rho}(a)} u \, \mathrm{d}x$, and $C = C (k)$.
\end{lemma}

\begin{proof}
We may assume that $M \equiv \| u \|_{L^\infty} = 1$.
In fact the case $M = 0$ is trivial, and for $0 < M < \infty$, first verify \eqref{caccioppoli} with $u$ replaced by $M^{-1} u$, and then multiply by $M^2$ to obtain the desired inequality
with the same constant $C$.

As \cite[Theorem 6.2]{HardtLin87} applies by the consideration in Lemma \ref{InteriorEnergyEstimate},
there exists $\omega \in H^1 \left ( B_s(a), \mathbf{D}_k \right )$, an extension of $u|_{\partial B_s(a)}$ to $B_s(a)$ satisfying
the estimate:
\begin{equation*}
\frac{1}{s} \int_{B_s(a)} \left | \nabla \omega \right |^2 \, \mathrm{d}x 
\leq \mu \int_{\partial B_s(a)} \left | \nabla_{\mathrm{tan}} u \right |^2 \, \mathrm{d}A
+ \frac{c}{\mu^2 s^2} \int_{\partial B_s(a)} \left | u - \bar{u} \right |^2 \, \mathrm{d}A,
\end{equation*}
for some $c = c( \left ( k, M \right ) > 0$ and any $\mu \in (0,1)$. 
Now by Fubini's theorem, for some $s \in \left (r, 2r \right )$, we have:
\begin{equation*}
\begin{split}
\int_{\partial B_s(a)} \left | \nabla_{\mathrm{tan}} u \right |^2 \, \mathrm{d}A \leq \frac{2}{r} \int_{B_{2r}(a)} | \nabla u |^2, \\
\int_{\partial B_s(a)} \left | u - \bar{u} \right |^2 \, \mathrm{d}A \leq \frac{1}{r} \int_{B_{2r}(a)} \left | u - \bar{u} \right |^2 \, \mathrm{d}x.
\end{split}
\end{equation*}
Then we obtain:
\begin{equation*}
\begin{split}
\frac{1}{r} \int_{B_r(a)} \left | \nabla u \right |^2 \, \mathrm{d}x
& \leq \frac{1}{s} \int_{B_{s}(a)} \left | \nabla u \right |^2 \, \mathrm{d}x \\
& \leq \frac{1}{s} \int_{B_{s}(a)} \left | \nabla \omega \right |^2 \, \mathrm{d}x \\
& \leq \frac{2 \mu}{2r} \int_{ B_{2r}(a)} \left | \nabla u \right |^2 \, \mathrm{d}x
+ \frac{8c}{\mu^2 (2r)^3} \int_{ B_{2r}(a)} \left | u - \bar{u} \right |^2 \, \mathrm{d}x,
\end{split}
\end{equation*}
where the first inequality is due to the monotonicity \eqref{harmonicmapmonotonicity}, the second inequality is due to the minimality of $u$, and the third
inequality is from the above estimate satisfied by the extension $\omega$ and $s^2 > r^2$. 
Choosing $\mu = \lambda/2$ and $C=32c$ yields the claim.
\end{proof}

Now we are ready to discuss the compactness and regularity of energy minimizing maps $u: \Omega \to \mathbf{D}_k$.

\begin{lemma} \label{compactness}
Let $u_i$ be a sequence of energy minimizing maps from $B_1(0)$ to $\mathbf{D}_k$, such that
$u_i \in L^{\infty} \left ( \partial B_1(0) \right )$  and
$u_i$ converges to $u$ weakly in $H^1 \left ( B_1(0), \mathbf{D}_k \right )$.
Then $u$ is an energy minimizing map on $B_1(0)$,
and $u_i$ converges to $u$ strongly in $H^1 \left ( B_1(0), \mathbf{D}_k \right )$.
\end{lemma}
\begin{proof}
We follow the argument of \cite[Proposition 5.1]{HKL88} closely.

Applying the Sobolev inequality to $W(2r)$ with $p = 6/5$, dividing each term by $r^2$, and applying \eqref{caccioppoli}, we obtain:
\begin{equation*}
\fint_{B_r(a)} \left | \nabla u_i \right |^2 \, \mathrm{d}x \leq \theta \fint_{B_{2r}(a)} \left | \nabla u_i \right |^2 \, \mathrm{d}x 
+ \frac{\tilde{C}}{\theta^2} \left ( \fint_{B_{2r}(a)} \left | \nabla u_i \right |^p \, \mathrm{d}x \right )^{2/p},
\end{equation*}
where $\theta \in (0,1)$, $p = 6/5 < 2$ and $\tilde{C} = \tilde{C}(k)$.
This is a reverse H\"older inequality with a small perturbation term. By \cite[Chapter V, Proposition 1.1]{Giaquinta83}, we conclude that  
that $\nabla u_i$ are equibounded in $L^q_{loc}$ for some $q > 2$. 

Let $w \in H^1 \left ( B_1(0), \mathbf{D}_k \right )$ be an arbitrary map with the boundary value $w = u$ on $\partial B_1(0)$.
Given $\delta > 0$, we choose a smooth cut-off function $\eta$ such that $\eta \equiv 1$ on $B_{1-\delta}(0)$, $\eta = 0$ on $\partial B_1(0)$, and
$| \nabla \eta | \leq \delta^{-1}$ on the annulus $A_{\delta} = B_1(0) \backslash B_{1-\delta}(0)$. First, we interpolate linearly and set:
\begin{equation}
 v_j = (1-\eta) u_j + \eta w \quad \mathrm{in} \; B_1(0). \label{interpolant}
\end{equation}
Note that $v_j$ agrees with $w$ on $\partial B_{1-\delta}(0)$ and with $u_j$ on $\partial B_1(0)$. However, it possibly doesn't have its image completely in the target $\mathbf{D}_k$.
But since $\mathbf{D}_k$ is simply connected, using the Lipschitz retraction constructed in Lemma \ref{InteriorEnergyEstimate} above 
and arguing as in the proof of \cite[Theorem 6.2]{HardtLin87},
we obtain a map $w_j \in H^1 \left ( A_\delta , \mathbf{D}_k \right )$ such that it satisfies the following:
\begin{equation}
\begin{split}
\int_{A_\delta} \left | \nabla w_j \right |^2 \, \mathrm{d}x \leq C \int_{A_\delta} \left | \nabla v_j \right |^2 \, \mathrm{d}x, \\
w_j = v_j = w \quad \mathrm{on} \; \partial B_{1-\delta}(0), \\
w_j = v_j = u_j \quad \mathrm{on} \; \partial B_1(0).
\end{split} \label{aExtension}
\end{equation}
Since we do not assume a uniform bound on $\left \| u_j \right \|_{L^\infty \left ( \partial B_1(0) \right ) }$, in fact we construct a Lipshitz retraction $P_j$ for each $u_j$.
Nevertheless, by the homogeneity of $\mathbf{D}_k$, we can ensure that the Lipschitz norms of $P_j$ are uniformly bounded. Hence, we can choose the constant $C = C(k)$ to be uniform in $j$.
Finally, we extend $w_j$ from $A_{\delta}$ to $B_1(0)$ by setting $w_j = w$ on $B_{1-\delta}(0)$. 

Given $\epsilon >0$, for $j$ large enough we have:
\begin{equation}
\int_{B_1(0)} | \nabla u |^2 \, \mathrm{d}x - \frac{\epsilon}{2} \leq  \int_{B_1(0)} \left | \nabla u_j \right |^2 \, \mathrm{d}x 
\leq \int_{B_1(0)} \left | \nabla w_j \right |^2 \, \mathrm{d}x, \label{comparison1}
\end{equation}
by the lower semicontinuity of the Dirichlet energy and the minimality of $u_j$. We estimate the right hand-side using \eqref{aExtension}:
\begin{equation}
\int_{B_1(0)} \left | \nabla w_j \right |^2 \, \mathrm{d}x \leq \int_{B_{1-\delta}(0)} | \nabla w |^2 \, \mathrm{d}x 
+ C \int_{A_\delta} \left | \nabla v_j \right |^2 \, \mathrm{d}x. \label{RHS1}
\end{equation}
Using \eqref{interpolant} and applying the Cauchy-Schwarz-Young inequality to the last term repeatedly, we obtain:
\begin{equation}
\int_{A_\delta} \left | \nabla v_j \right |^2 \, \mathrm{d}x \leq 
c_1 \int_{A_\delta} \left | \nabla u_j \right |^2 \, \mathrm{d}x +
c_2  \int_{A_\delta} \left | \nabla w \right |^2 \, \mathrm{d}x + 
c_3 \delta^{-2} \int_{A_\delta} \left | u_j - w \right |^2 \, \mathrm{d}x. \label{RHS2}
\end{equation}
We apply the H\"older inequality to the first term, and use the equiboundedness of $\nabla u_j$ in $L^q$ for some $q>2$ to get:
\begin{equation*}
\int_{A_\delta} \left | \nabla u_j \right |^2 \, \mathrm{d}x \leq b_1 \left | A_{\delta} \right |^{1-2/q}. 
\end{equation*}
To the last term in \eqref{RHS2}, we apply the triangle and Young inequalities to get:
\begin{equation*}
\delta^{-2} \int_{A_\delta} \left | u_j - w \right |^2 \, \mathrm{d}x \leq
2 \delta^{-2} \int_{A_\delta} \left | u_j - u \right |^2 \, \mathrm{d}x + 2 \delta^{-2} \int_{A_\delta} \left | u - w \right |^2 \, \mathrm{d}x.
\end{equation*}
Then noting $u - w = 0$ on $\partial B_1(0)$, to the second term we apply the variant of Poincar\'e inequality for functions that vanish on a subset of the boundary with
a certain measure. Picking up a constant factor $b_2 \delta^2$ from the Poincar\'e inequality, we get:
\begin{equation*}
\begin{split}
\int_{A_\delta} \delta^{-2} \left | u_j - w \right |^2 \, \mathrm{d}x & \leq 
2 \delta^{-2} \int_{A_\delta} \left | u_j - u \right |^2 \, \mathrm{d}x +
2 \delta^{-2} b_2 \delta^2  \int_{A_\delta} \left | \nabla (u-w) \right |^2 \, \mathrm{d}x \\
& \leq 2 \delta^{-2} \int_{A_\delta} \left | u_j - u \right |^2 \, \mathrm{d}x + b_3  \int_{A_\delta} \left | \nabla (u-w) \right |^2 \, \mathrm{d}x.
\end{split}
\end{equation*}
Thus, we obtain:
\begin{equation*}
\begin{split}
\int_{A_\delta} \left | \nabla v_j \right |^2 \, \mathrm{d}x \leq
& \tilde{C_1} \left | A_{\delta} \right |^{1-2/q} +
\tilde{C_2} \delta^{-2} \int_{A_\delta} \left | u_j - u \right |^2 \, \mathrm{d}x \; + \\
& \tilde{C_3} \int_{A_\delta} \left | \nabla w \right |^2 \, \mathrm{d}x + 
\tilde{C_4} \int_{A_\delta} \left | \nabla (u-w) \right |^2 \, \mathrm{d}x.
\end{split}
\end{equation*}
Firstly, letting $j \to \infty$, as $u_j$ converge to $u$ in $L^2 \left ( B_1(0) \right )$ strongly, the second term above converges to zero.
Next, choosing $\delta > 0$ small enough and using the absolute continuity of the integrals in terms three and four above, we obtain:
\begin{equation}
\lim \sup_{j \to \infty} \int_{A_\delta} \left | \nabla v_j \right |^2 \, \mathrm{d}x \leq \frac{\epsilon}{2C}. \label{interpolantEst}
\end{equation}
Together with \eqref{comparison1} and \eqref{RHS1}, this estimate gives:
\begin{equation*}
\int_{B_1(0)} | \nabla u |^2 \, \mathrm{d}x \leq \int_{B_1(0)} | \nabla w |^2 \, \mathrm{d}x + \frac{\epsilon}{2} + \frac{\epsilon}{2} 
= \int_{B_1(0)} | \nabla w |^2 \, \mathrm{d}x + \epsilon.
\end{equation*}
Since the map $w$ and $\epsilon > 0$ are arbitrary, we conclude that $u$ is a minimizer.

In order to observe the strong convergence, we expand:
\begin{equation*}
\int_{B_1(0)} \left | \nabla \left ( u_j - u \right ) \right |^2 \mathrm{d}x =
\int_{B_1(0)} \left | \nabla u_j \right |^2 \mathrm{d}x +
\int_{B_1(0)} \left | \nabla u \right |^2 \mathrm{d}x -
2 \int_{B_1(0)} \left \langle \nabla u_j , \nabla u \right \rangle \mathrm{d}x,
\end{equation*}
and note that for any given $\epsilon > 0$, by the $H^1$-weak convergence $u_j$, there exists a large enough $j$ such that:
\begin{equation*}
\begin{split}
\int_{B_1(0)} \left | \nabla \left ( u_j - u \right ) \right |^2 \mathrm{d}x & \leq
\int_{B_1(0)} \left | \nabla u_j \right |^2 \mathrm{d}x +
\int_{B_1(0)} \left | \nabla u \right |^2 \mathrm{d}x -
2 \int_{B_1(0)} \left | \nabla u \right |^2 \mathrm{d}x + \frac{\epsilon}{2} \\
& = \left [ \int_{B_1(0)} \left | \nabla u_j \right |^2 \, \mathrm{d}x - \int_{B_1(0)} \left | \nabla u \right |^2 \, \mathrm{d}x \right ] + \frac{\epsilon}{2}.
\end{split}
\end{equation*}
On the other hand, taking $w = u$ in \eqref{interpolant}, by the minimality of $u_j$, we have:
\begin{equation*}
\int_{B_1(0)} \left | \nabla u_j \right |^2 \, \mathrm{d}x \leq \int_{B_1(0)} \left | \nabla w_j \right |^2 \, \mathrm{d}x 
\leq \int_{B_1(0)} \left | \nabla u \right |^2 \, \mathrm{d}x + C \int_{A_\delta} \left | \nabla v_j \right |^2 \, \mathrm{d}x.
\end{equation*}
Hence, by \eqref{interpolantEst} we conclude:
\begin{equation*}
\lim \sup_{j \to \infty} \int_{B_1(0)} \left | \nabla \left ( u_j - u \right ) \right |^2 \, \mathrm{d}x \leq \epsilon,
\end{equation*}
that is we have strong convergence in $H^1 \left ( B_1 (0) \right )$.

\end{proof}

\begin{lemma} \label{holderestimate}
Let $u \, : \, \Omega \to \mathbf{D}_k$ be a minimizing map. 
Then there exists an $\alpha = \alpha \left ( k, \| u \|_{L^2 \left ( {\Omega} \right )} \right )$ such that for any compact $K \subset \Omega$,
the following estimate holds:
\begin{equation}
\| u \|_{C^{0, \alpha}(K)} \leq C(K, \Omega, k) \| u \|_{L^2 \left ( {\Omega} \right )}.
\label{uniformholder}
\end{equation}
\end{lemma}
\begin{proof}
We present a proof in the spirit of \cite{Wang2001}. 
Since the result is local, it suffices to prove \eqref{uniformholder} for a neighborhood of $0 \in \Omega$.
We will also restrict our attention to the physical case $d=3$ for simplicity, but the argument is identical for any dimension.

Firstly, we claim that for every $\epsilon > 0$, there exists a $\theta = \theta \left ( M, \epsilon \right )$ such that $E(0;\theta) < \epsilon$,
where $M = \| u \|_{L^{\infty}\left ( \partial B_1(0) \right )}$.
If not, there is an $\epsilon_0 > 0$ such that $\epsilon_0 \leq E(0;\theta)$ for every $\theta > 0$. 
From \eqref{subharmonicaverage} and \eqref{harmonicmapmonotonicity} we have:
\begin{equation*}
\begin{aligned}
\epsilon_0 \leq E(0; \theta) \leq \frac{1}{2 \left ( \log ( 1/ \theta ) \right )} \int_{\theta}^1 \frac{2 E(0;r)}{r} \, \mathrm{d}r
& \leq \frac{H(0;1) - \theta^{-2} H(0;\theta)}{2 \log ( 1/ \theta ) } \\
& \leq  \frac{ M^2 \mathcal{H}^2 \left ( \partial B_1(0) \right ) }{2  \log (1/\theta) } .
\end{aligned}
\end{equation*}
Hence, choosing $\theta$ small enough we obtain a contradiction. 

Next, we claim that there exist positive constants $\epsilon_0 > 0$ and $\tilde{\theta} > 0$, depending only on $k$ and $M$ such that if $E(0;1) < \epsilon_0$, then:
\begin{equation}
E(0;\theta R) \leq \frac{1}{2} E(0; R),
\label{energydecay} 
\end{equation}
From \eqref{energydecay}, it is standard to deduce that the renormalized Dirichlet energy $E(0;r)$ decays as a power of $r$, as $r \to 0^+$, and then
by Morrey's lemma that $u$ is locally H\"older continuous. Moreover, by scaling we may assume $R=1$. 

To prove \eqref{energydecay} by contradiction, we consider a sequence of minimizers $\left \{ u_i \right \}$ such that:
\begin{enumerate}[(i)]
 \item $\left | u_i \right | \leq M$,
 \item $E_i (0;1) = \int_{B_1(0)} \left | \nabla u_i \right |^2 \, \mathrm{d}x \to 0$, as $i \to \infty$, \label{renormalization}
and
 \item $E_i \left ( 0; \tilde{\theta} \right ) > (1/2) E_i (0;1)$, for $\tilde{\theta}$ to be chosen. \label{contradecay}
\end{enumerate}

If there exists a $C > 0$ such that $H_i(0;1) \leq C E_i(0;1)$ for all $i$, then \eqref{subharmonicaverage} and \eqref{harmonicmapmonotonicity} give:
\begin{equation*}
2 E_i \left (0; \tilde{\theta} \right ) \cdot \log \left (1/ \tilde{\theta} \right ) \leq \int_{\theta}^1 \frac{2 E_i(0;r)}{r} \, \mathrm{d}r \leq  H_i(0;1) \leq C E_i(0;1).
\end{equation*}
Choosing $\tilde{\theta} \in \left ( 0, e^{-C} \right )$, we obtain \eqref{energydecay} without using the small energy hypothesis. 
Therefore, we can restrict our attention to the case $H_i(0;1)$ does not decay as fast as $E_i(0;1)$. 
In particular, using the scaling invariance of the target $\mathbf{D}_k$, we can multiply $u_i$ by $H_i(0;1)^{-1/2}$ to assume:
\begin{equation}
H_i(0;1) = \mathcal{H}^2 \left ( \partial B_1(0) \right ), \label{bnormalization}
\end{equation}
while \eqref{renormalization} and \eqref{contradecay} still hold for the rescaled minimizers $\left \{ u_i \right \}$. 
Hence, by Lemma \ref{compactness}, $u_i \to u_\infty$ weakly in $H^1 \left ( B_1(0) \right )$ and strongly in $L^2 \left ( B_1(0) \right ) \cap H^1_{loc} \left ( B_1(0) \right )$,
where $u_\infty$ is a constant map, because its energy is zero by weak lower semicontinuity. 

For each $i$, we let $\tilde{u}_i \in \mathbf{D}_k$ be a minimizer of $\mathrm{dist} \left ( \overline{u_i} , \mathbf{D}_k \right )$, 
where $\overline{u_i} = \fint_{B_1(0)} u_i \, \mathrm{d}x $.
We note that:
\begin{equation}
\mathrm{dist}^2 \left ( \overline{u_i} , \mathbf{D}_k \right ) = \fint_{B_1(0)} \left | u_i - \overline{u_i} \right |^2 \, \mathrm{d}x \leq C E_i(0;1)
\label{convergingAv}
\end{equation}
by the Poincar\'e inequality. Likewise:
\begin{equation*}
0 \leq \fint_{B_1(0)} \left ( \left | u_i \right |^2 - \left | \overline{u_i} \right |^2 \right ) \, \mathrm{d}x
= \fint_{B_1(0)} \left | u_i - \overline{u_i} \right |^2 \, \mathrm{d}x \leq C E_i(0;1),
\end{equation*}
while by \eqref{Hcomparison}:
\begin{equation*}
\fint_{B_1(0)} \left | u_i \right |^2 \, \mathrm{d}x \geq \frac{1}{3 + 2 N_i(0;1)} \fint_{\partial B_1(0)} \left | u_i \right |^2 = \frac{1}{3+2 E_i(0;1)}.
\end{equation*}
Therefore, for $i$ large enough:
\begin{equation} 
\left | \overline{u_i} \right |^2 \geq \frac{1}{3+2 E_i(0;1)} - C E_i(0;1) \geq 1/6.
\label{awayAv}
\end{equation}
Similarly, from Jensen's inequality and \eqref{Hcomparison}, we have:
\begin{equation*} 
\left | \overline{u_i} \right |^2 \leq \mathcal{H}^2 \left ( \partial B_1(0) \right ) / \left | B_1(0) \right |.
\end{equation*}
Hence, passing to a subsequence if necessary, $\overline{u_i}$, and by \eqref{convergingAv} $\tilde{u}_i$, both converge to $u_*$, a constant with $\left | u_* \right |^2 \geq 1/6$.

We let $r_i = \left | u_* - \tilde{u}_i \right |$. If $\lim_{i \to \infty} \left (r_i / \epsilon_i \right ) < \infty$, then we define: 
$$U_i = E_i(0;1)^{-1/2} \left ( u_i - u_* \right ),$$
which satisfies $\left \| U_i \right \|_{L^2 \left ( B_1(0) \right ) } \leq C$, by the triangle and Poincar\'e inequalities, 
and $\left \| \nabla U_i \right \|_{L^2 \left ( B_1(0) \right ) } = 1$. Hence, by the $H^1$-weak convergence $U_i \rightharpoonup U_\infty$, $U_\infty$ maps 
almost every point in $B_1(0)$ to $T_{u_*} \mathbf{D}_k$, which is a hyperplane, as $\left | u_* \right |^2 \geq 1/6$.

If $\lim_{i \to \infty} \left (r_i / \epsilon_i \right ) < \infty$ is not true, then we define instead:
$$U_i = E_i(0;1)^{-1/2} \left ( u_i - \tilde{u}_i \right ).$$
We observe that $\left \| \nabla U_i \right \|_{L^2 \left ( B_1(0) \right ) } = 1$, and $\left \| U_i \right \|_{L^2 \left ( B_1(0) \right ) } \leq C$. The latter estimate is due to
\eqref{convergingAv} and the Poincar\'e inequality. Hence, $U_i \rightharpoonup U_\infty$ weakly in $H^1 \left (B_1(0) \right )$ and strongly in $L^2 \left (B_1(0) \right )$. 
Once again we claim that $U_\infty$ maps $B_1(0)$ to $T_{u_*} \mathbf{D}_k$ almost everywhere. 

To verify this claim, firstly, we note that due to the strong convergence in 
$L^2 \left ( B_1(0) \right )$ and Egorov's theorem, for every $\delta > 0$, there exists a $E_\delta \subset B_1(0)$ such that $\left | B_1(0) \backslash E_\delta \right | < \delta$, 
$U_i$ and $U_\infty$ are bounded on $E_\delta$, and the convergence is uniform.
Secondly, there exists a sequence of maps $W_i : B_1(0) \setminus E_\delta \to T_{\tilde{u}_i} \mathbf{D}_k$ such that $ \left | U_i - W_i \right | = \gamma_i \to 0$.
Finally, since \eqref{awayAv} guarantees that $\tilde{u}_i$ and $u_*$ are bounded uniformly away from $0$,
there exists a sequence of map $\tau_i : \, T_{\tilde{u}_i} \mathbf{D}_k \to T_{\tilde{u}_*} \mathbf{D}_k$, which coverges to the identity, as $i \to \infty$.
Letting $\tilde{W}_i = \tau_i \circ W_i : B_1(0) \setminus E_\delta \to T_{\tilde{u}_*} \mathbf{D}_k$, we have:
\begin{equation*}
\begin{aligned}
\left | \tilde{W}_i - U_\infty \right | & \leq \left | U_\infty - U_i \right | + \left | U_i - W_i \right | + \left | W_i - \tilde{W}_i \right | \\
& \leq  \left | U_\infty - U_i \right | + \left | U_i - W_i \right |
+ \left | \tau_i - id \right | \cdot \left ( \left | U_i \right | + \gamma_i \right ).
\end{aligned}
\end{equation*}
Thus, we conclude that $U_\infty$ maps $B_1(0) \backslash E_\delta$ to $T_{u_*} \mathbf{D}_k$. But since $\delta > 0$ is arbitrary, the claim must hold almost everywhere. 
Since $u_* \neq 0 \in \mathbf{D}_k$, $U_\infty$ takes value in a hyperplane.

Our final claim is that $U_\infty$ is a vector-valued harmonic function. 
$U_i$ also satisfy the Caccioppoli-type inequality \eqref{caccioppoli} with the same uniform constant as $u_i$. 
Furthermore, the finite energy extension in Lemma \ref{compactness}, applies with a uniform estimate, even though the targets of $U_i$ differ by translation and scaling. 
To see this, we recall that the constant in \eqref{aExtension} depends only on the Lipschitz norms of the retractions $P_i$ constructed for each $U_i$. 
Since $\tilde{u}_i$ are bounded uniformly away from $0 \in \mathbf{D}_k$, the Lipschitz norms of $P_i$ are uniformly bounded,
i.e. we can proceed exactly as in the case of a smooth target. 
Thus, repeating the argument in Lemma \ref{compactness}, we conclude that the limit $U_\infty$ is a minimizer.

However, since it takes value in a hyperplane, $U_\infty$ must be a vector-valued harmonic function. 
In addition, the convergence is strong in $H^1_{loc} \left ( B_1(0) \right )$.
In particular, $U_\infty$ satisfies the following basic estimate for harmonic functions:
\begin{equation*}
\int_{B_{\tilde{\theta}}(0)} \left | \nabla U_\infty \right |^2 \, \mathrm{d}x \leq C \tilde{\theta}^3 \int_{B_1(0)} \left | \nabla U_\infty \right |^2 \, \mathrm{d}x. 
\end{equation*}
From the strong convergence of $U_i$ to $U_\infty$ in $H^1_{loc} \left ( B_1(0) \right )$ 
and the fact that 
$$ \int_{B_1(0)} \left | \nabla U_\infty \right |^2 \, \mathrm{d}x \leq 1,$$ 
we obtain:
\begin{equation*}
\lim_{i \to \infty} \left ( \frac{1}{\tilde{\theta}} \int_{B_{\tilde{\theta}}} \left | \nabla U_i \right |^2 \, \mathrm{d}x \right ) \leq C \tilde{\theta}^2 = \frac{1}{3},
\end{equation*}
for $\tilde{\theta} \leq \sqrt{1/3C}$. Hence, for $i$ large enough, we obtain:
\begin{equation*}
\frac{1}{\tilde{\theta}} \int_{B_{\tilde{\theta}}(0)} \left | \nabla U_i \right |^2 \, \mathrm{d}x < \frac{1}{2} = \frac{1}{2} \int_{B_1(0)} \left | \nabla U_i \right |^2 \, \mathrm{d}x.
\end{equation*}
contradicting \eqref{contradecay}. 

We conclude that \eqref{energydecay} holds for some $\epsilon_0$ and $\tilde{\theta} > 0$. 
Moreover, due to \eqref{subharmonicity}, the mean-value inequality for subharmonic functions, and a standard covering argument, the $L^\infty$-norm of $u$ in every compact
$E \subset \Omega$ is controlled by $\| u \|_{L^2(\Omega)}$. Thus, $\tilde{\theta}$ depends only on the latter, as well as $k$.  
\end{proof}

\begin{remark} \label{higherregularity}
Having established the continuity of energy minimizing maps $u$ into $\mathbf{D}_k$ in the interior of $\Omega$, we can recover the variable order parameter $s$ 
and the $\mathbf{RP}^2$-valued director field $n$ on the set $\left \{ x \in \Omega \, : \, \left | u(x) \right | > 0 \right \}$ as:
\begin{equation*}
s = k^{-1/2} |u|, \quad  u = \left ( \sqrt{k-1}s, v \right ), \quad n = \frac{1}{s} v. 
\end{equation*}
Moreover, the analyticity of the energy minimizing map $u$ on the complement of $u^{-1} \{0\}$ follows from the fact that $\mathbf{D}_k$ is analytic away from its vertex. 
Therefore, the corresponding $s$ and $n$ are also analytic on the complement of the closed set $\left \{ x \in \Omega \, : \, s(x) = 0 \right \}$ in $\Omega$.
See \cite[Section 3.2]{Lin89} for the proof that $\mathrm{sing}(n) = s^{-1} \{0\}$.
\end{remark}

We note that the monotonicity of the Almgren frequency
\eqref{almgrenmonotonicity}, the compactness of minimizers as in Lemma \ref{compactness} and the regularity theory as in Lemma \ref{holderestimate}
allow us to consider the blow-up sequences of minimizers whose limits are non-trivial homogeneous minimizers. In other words, we have the following lemma:

\begin{lemma} \label{existenceHblowups}
For $u \, : \, \Omega \to \mathbf{D}_k$ an energy minimizing map such that $u(a) = 0$, and for $r_i \in \left ( 0 , \frac{1}{2} \mathrm{dist}(a, \partial \Omega) \right )$,
the sequence of maps
\begin{equation} 
u_i(x) = \left ( \fint_{B_{r_i}(a)} |u|^2 \, \mathrm{d}x \right )^{-\frac{1}{2}} u \left ( a + r_i x \right )
\label{scaling}
\end{equation}
has a subsequence that converge strongly in $H^1 \left ( B_2(0) \right )$ and uniformly on $B_{2}(0)$ to a non-zero minimizing map $\varphi$, which is homogeneous
of degree $N_u \left ( a ; 0^+ \right )$.
\end{lemma}

\begin{proof}
Firstly, we derive a uniform $H^1$ bound for $u_i$.
\begin{equation*}
\int_{B_2(0)} \left | \nabla u_i \right |^2 \, \mathrm{d}x = 
\frac{ \left ( 1 / r_i \right ) \int_{B_{2r_i}(a)} | \nabla u |^2 \, \mathrm{d}x }{\fint_{B_{r_i}(a)} |u|^2 \, \mathrm{d}x }
= N \left ( a ; 2r_i \right ) \cdot \left [ \left ( r_i / 2 \right ) \frac{\int_{\partial B_{2r_i}(a)} |u|^2 \, \mathrm{d}A}{\int_{B_{r_i}(a)} |u|^2 \, \mathrm{d}x} \right ].
\end{equation*}
Using \eqref{Hcomparison} as in the derivation of \eqref{localFcontrol} to estimate the denominator from below in terms of the numerator in the second factor, we obtain:
\begin{equation*}
\int_{B_2(0)} \left | \nabla u_i \right |^2 \, \mathrm{d}x \leq N \left ( a; 2r_i \right ) \left ( 3 + 2 N \left ( a; 2r_i \right ) \right ) 2^{2+N \left ( a; 2r_i \right )}. 
\end{equation*}
Hence, by the monotonicity formula \eqref{almgrenmonotonicity}, we have a uniform $H^1 \left ( B_2 (0) \right )$ bound for $u_i$.

Next we observe that since $u \left ( a + r_i x \right )$ are minimizers, uniformly bounded on $B_2(0)$, they satisfy the Caccippoli inequality \eqref{caccioppoli} with a uniform constant. 
Since each term is quadratic in this inequality, dividing both sides by the denominator of \eqref{scaling}, we obtain a corresponding Caccippoli inequality with the same
uniform constant that is satisfied by $u_i$. Moreover, each $w_i$ is a minimizer. Hence by Lemma \ref{compactness}, $\varphi$ is a minimizer and the convergence is strong
$H^1$. Observing that $\left \| u_i \right \|_{L^2 \left (B_1(0) \right )} = 1$, from \eqref{Hcomparison} and \ref{holderestimate} we infer that the convergence is also
uniform in $B_2(0)$. Also note that this limit is not the zero map, since $ \left \| \varphi \right \|_{L^2 \left (B_1(0) \right )} = 1 $.  

Finally, using the scaling property of and monotonicity of the Almgren frequency, as well as the mode of the convergence, we observe that:
\begin{equation}
N_{\varphi} (0; \rho) = \lim_{i \to \infty} N_{u_i}(0;\rho) = \lim_{i \to \infty} N_{u}(a; r_i\rho) = N_u \left ( a ; 0^+ \right ),
\label{frequencyandlimit}
\end{equation}
for every $\rho \in (0,1)$. Hence, we conclude that $\varphi$ is homogeneous of degree $N_u \left ( a ; 0^+ \right )$ by Lemma \ref{monotonocity}.
\end{proof}

Finally we recall the estimate on the Hausdorff dimension of the zero set of $u$:

\begin{lemma} \label{Hdimension}
For any energy minimizing map $u \, : \, \Omega \to \mathbf{D}_k$ with $k>1$, $u^{-1} \{0\}$ is either all of $\Omega$ or, it has Hausdorff dimension less than or equal to 1.
\end{lemma}
\begin{proof}
The proof in \cite{Lin89} and \cite{Lin91} is based on the dimension reduction principle as carried out in \cite{Lin91N}. We outline the argument here for completeness. 

Note that if $F = u^{-1} \{0\} \cap B_1(0)$ for a minimizing map $u$, and $B_1(0) \subset \Omega$, then $F$ is relatively closed by Lemma \ref{holderestimate}.
We observe that the following two properties also hold: Firstly, the collection of zero sets of minimizing maps are closed under scaling and translations.
Secondly, for the zero set of every minimimizing map there exists a homogeneous degree zero ``tangent set'', cf. \cite[Section 2]{Lin91N} for details.
The second property is a consequence of the monotonicity of the Almgren frequency \eqref{almgrenmonotonicity}, the compactness of minimizers as in Lemma \ref{compactness} and
the estimate in Lemma \ref{holderestimate}. Consequently, the dimension reduction principle applies to the zero sets of minimizing maps.

Therefore, by the argument sketched in \cite{Lin91N}, either $s \equiv 0$ on all of $\Omega$, or $\mathrm{dim}_{H} \left ( s^{-1}\{0 \} \right ) \leq d$ for some nonnegative integer
$d \leq 2$. In order to show $d < 2$, it suffices to rule out minimizers depending on one variable, which are essentially minimizing geodesics of constant speed. 
Observe that a minimizing geodesic of constant speed cannot hit the vertex of $\mathbf{D}_k$ at an interior point of its domains without being trivial. 
Hence, we obtain the desired Hausdorff dimension estimate.
\end{proof}

\begin{remark} \label{linedefect}
The Hausdorff dimension estimate in \eqref{Hdimension} is sharp, as pointed out in \cite[Section 4]{HardtLin93}.
Consider $\Omega = \mathbf{B}_1^{2}(0) \times (0,1)$ with the boundary data $n_0 \left ( x_1, x_2, x_3 \right ) =  \left [ \left ( x_1, x_2, 0 \right ) \right ] \in \mathbf{RP}^2$
and $s_0 > 0$
on $\partial \mathbf{B}_1^2(0) \times (0,1)$. By Lemma \ref{holderestimate}, the corresponding minimizer $u$ is continuous. 
However, $n_0$ does not have a continuous extension to $\mathbf{B}_1^{2}(0) \times \left \{ t \right \}$ for any $t \in (0,1)$. 
Therefore, on each horizontal slice $s$ must vanish. Hence,
$\mathrm{dim}_{H} \left ( s^{-1}\{0 \} \right ) \geq 1$.
\end{remark}

\section{Homogeneous Minimizers} \label{HomogeneousMinimizers}

In this section we consider non-constant homogeneous energy minimizing maps $v \, : \mathbf{R}^m \to \mathbf{D}_k$
for $m = 1, 2, 3$. 
By Lemma \ref{existenceHblowups} such maps arise as the blow-up limits of general minimizers at points in their zero sets. 
The cases $m = 1,2$ relate to tangent maps that are independent of two variables and one variable respectively. 
We first state a lemma that follows from \cite{Lin89}, \cite{Lin91} and \cite{HardtLin93}.

\begin{lemma} \label{lowerHomogeneous}
Let $w \, : \mathbf{R}^m \to \mathbf{D}_k$ be a homogeneous energy minimizing map with $k>1$. 
If $m=1$, then $w$ has to be a constant map.
If $m=2$, then $w$ is determined by the formula:
\begin{equation}
w \left ( r e^{i \theta} \right ) = r^{1/2\sqrt{k}} \left (\sqrt{k-1}, \left [ (e^{i \phi/2},0 ) \right ] \right ),
\label{2Dhminimizer} 
\end{equation}
uniquely up to rotations and scalar multiplication by constants.
\end{lemma}

\begin{proof}
The case $m=1$ is already addressed in Lemma \ref{existenceHblowups}.
In the case $m=2$ we give an overview of the corresponding classification result in \cite[Section 4]{HardtLin93}, for completeness.

Since $w$ is homogeneous, we can write $w(r, \theta) = r^{\alpha} \varphi ( \theta )$, 
where $\alpha > 0$ and $\varphi : \mathbf{R} \to \mathbf{D}_k$ is absolutely continuous and $2 \pi$-periodic.
Moreover, $\varphi = \left ( \sqrt{k-1} | \psi | , [ \psi ] \right )$, where $\psi : \mathbf{R} \to \mathbf{R}^3$.
Considering the variations of $w$ given by
\begin{equation*}
w^t = r^\alpha \left ( \sqrt{k-1} \left | \psi^t \right |, \left [ \psi^t \right ] \right ), 
\end{equation*}
where $\psi^t (r, \theta ) = \psi (\theta) + t \xi (r, \theta )$ for $\xi$ smooth, $2 \pi$-periodic in $\theta$ and compactly supported,
we obtain the following Euler-Lagrange equation:
\begin{equation}
\psi_{\theta \theta} + \left  [k \alpha^2 + (k-1) Q_\theta + (k-1) Q^2 \right ] \psi = 0,
\label{ELhom} 
\end{equation}
which is valid on intervals on which $\psi$ does not vanish. 
Here $Q = | \psi |^{-2} \psi \cdot \psi_\theta$, and \eqref{ELhom} can be written down via a suitable (local) lifting.

Next, in order to simplify this nonlinear ordinary differential equation, we make use of the fact that in the case $m = 2$,
the Hopf differential of the map $w$ defined as:
\begin{equation}
\omega^w = \left ( \left | w_x \right |^2 - \left | w_y \right |^2 \right ) + 2 \left ( w_x \cdot w_y \right ) i
\label{Hopf} 
\end{equation}
is holomorphic. As $\omega^w = f(z) dz^2$ for some entire function $f$ of the form $r^{2 \alpha - 2} \Phi(\theta)$,
we have: either $\omega^w \equiv 0$, or $f(z) = Cz^{2\alpha - 2}$ for $\alpha = n/2$ and some integer $n \geq 2$, and some non-zero complex scalar $C$.

In the first case, using $\omega^w \equiv 0$, it is easy to check that $| \psi |^2$ is a positive constant, $Q \equiv 0$, and \eqref{ELhom} reduces to:
\begin{equation}
 \psi_{\theta \theta} + \left ( k \alpha^2 \right ) \psi = 0,
\label{ELhom2}
\end{equation}
which now holds for all $\theta$, as $ | \psi | \equiv \lambda > 0$. Finally, checking that $\psi$ maps into a 2-dimensional subspace,
we obtain after adjusting the coordinates for $\mathbf{R}^3$ with a suitable rotation:
$ \psi = \lambda \left ( \cos ( n \theta ), \sin ( n \theta ), 0 \right ) $ for $n^2 = k \alpha^2$. 

Here at a first inspection $n$ can be any positive half-integer, since we only require $[ \psi ]$ to be $2 \pi$-periodic and 
$v$ to be in $C^{0,\alpha}_{loc}$ for some $\alpha >0$ by Lemma \ref{holderestimate}. 
However, when $|n| > 1$, for each integer $j$, 
the closed curve that is the image of $\left [ 0, \frac{2 \pi j}{n} \right ]$ 
under the map $[ \psi ]$ covers its image with even multiplicity. But since $\pi_1 \left ( \mathbf{RP}^2 \right ) = \mathbf{Z}_2$, any such
curve is contractible away from the origin. When the curve is contractible, the construction in \cite[Section 2]{HardtLin93}, for the case $\varphi$
mapping into a cone over $\mathbf{S}^2$ instead of $\mathbf{RP}^2$ carries over. 
A minimizing map $w$ restricted to the infinite wedge given by $\left [ 0, \frac{2 \pi j}{n} \right ]$ is minimizing on this subdomain. 
On the other hand, the variation constructed in \cite[Section 2]{HardtLin93}, mapping outside the subcone generated by the image of $\psi$ decreases the energy, 
giving a contradiction.
Hence, we must have $n = 1/2$. Thus, we arrive at \eqref{2Dhminimizer}.

While the second case $\omega^w \not \equiv 0$ is more complicated, the requirement that the Hopf differential \eqref{Hopf} is holomorphic and the equation
\eqref{ELhom} lead to modifications of the corresponding maps (2) and (3) in \cite[Theorem 3.2]{HardtLin93}, which necessarily map infinite wedges in $\mathbf{R}^2$ 
onto subcones with even multiplicity. Such maps fail to be minimizers in these subdomains by the above ``peeling-off'' argument. 
Therefore, they cannot be global minimizers either. 
Hence, \eqref{2Dhminimizer} is the only remaining candidate for a minimizer. 
Since we know that there exists a minimizer for the boundary data that is the restriction of \eqref{2Dhminimizer} to the unit circle, 
we conclude that \eqref{2Dhminimizer} is indeed a minimizer.
\end{proof}

\begin{remark} \label{2Dfrequency}
The $L^2$-normalization of the blow-up sequences considered in Lemma \ref{existenceHblowups} eliminates scalar multiplication by constants,
leaving the group of rotations in $\mathbf{R}^3$ as the only source of possible non-uniqueness for tangent maps of two variables. 
Thus, the question becomes whether two distinct sequences of $r_i \downarrow 0$ could give rise to two distinct blow-up limits that differ by some rotation, cf. \cite{AUniqueness}.
\end{remark}

Next we address the case $m=3$. Firstly, we prove a local almost-minimality result.

\begin{lemma} \label{almostminimizer}
For any non-constant energy-minimizing map $v: \mathbf{R}^3 \to \mathbf{D}_k$ that is homogeneous of degree $\alpha$,  
$v|_{\mathbf{S}^2}$  is locally almost minimizing in the sense that there exists a constant $C$ 
such that for all $b \in \mathbf{S}^2$ and $\rho > 0$:
\begin{equation}
\int_{\mathbf{S}^2 \cap B_{\rho}(b)} \left | \nabla_{\mathrm{tan}} v \right |^2 \, \mathrm{d}\mathcal{H}^2
\leq C \rho +
\int_{\mathbf{S}^2 \cap B_{\rho}(b)} \left | \nabla_{\mathrm{tan}} h \right |^2 \, \mathrm{d}\mathcal{H}^2, 
\label{almostminimality}
\end{equation}
whenever $h \in H^1 \left ( \mathbf{S}^2 \cap B_{\rho}(b), \mathbf{D}_k \right )$ agrees with $v$ on $\mathbf{S}^2 \cap \partial B_{\rho}(b)$.
The constant $C$ depends on $\alpha$, $E_0 = \int_{\mathbf{S}^2} | \nabla_{\mathrm{tan}} v |^2 \, \mathrm{d} x $ 
and $ M_0 =   \| v \|_{L^\infty \left ( S^2 \right )}^2 $ only. Moreover, if $v$ is replaced with $\lambda v$ for a constant $\lambda > 0$, then
\eqref{almostminimality} holds with $\lambda^2 C$ instead of $C$.
\end{lemma}

\begin{proof}
Note that the claim holds trivially when $\rho \geq 1/2$, or when
\begin{equation*}
\int_{\mathbf{S}^2 \cap B_{\rho}(b)} \left | \nabla_{\mathrm{tan}} h \right |^2 \, \mathrm{d}\mathcal{H}^2 \geq 
\int_{\mathbf{S}^2 \cap B_{\rho}(b)} \left | \nabla_{\mathrm{tan}} v \right |^2 \, \mathrm{d}\mathcal{H}^2.
\end{equation*}
Hence, we assume $\rho < 1/2$, and
\begin{equation*}
\int_{\mathbf{S}^2 \cap B_{\rho}(b)} \left | \nabla_{\mathrm{tan}} h \right |^2 \, \mathrm{d}\mathcal{H}^2 < 
\int_{\mathbf{S}^2 \cap B_{\rho}(b)} \left | \nabla_{\mathrm{tan}} v \right |^2 \, \mathrm{d}\mathcal{H}^2 \leq E_0.
\end{equation*}
By Lemma \ref{maximumprinciple} we may also assume that $| h | \leq M_0$ on $\mathbf{S}^2 \cap B_{\rho}(b)$.

Recall that $\alpha > 0$ by Lemma \ref{holderestimate}. Consider $\tilde{h}$, the extension of $h$ to the conical domain
\begin{equation*} 
C_{\rho} = \left \{ (t,tx) \in \mathbf{R} \times \mathbf{R}^3 \, : \, t \in [- \rho + 1/2,3/2 + \rho], \; x \in \mathbf{S}^2 \cap B_{\rho}(b) \right \},
\end{equation*}
as follows:
\begin{equation}
\tilde{h}(t,x) = \left \{
\begin{aligned}
t^{\alpha} h(x), \quad  & t \in [1/2, 3/2], \, x \in \mathbf{S}^2 \cap B_{\rho}(b), \\
h_1(t,x), \quad & t \in [-\rho + 1/2, 1/2], \, x \in \mathbf{S}^2 \cap B_{\rho}(b), \\
h_2(t,x), \quad & t \in [3/2, 3/2 + \rho], \, x \in \mathbf{S}^2 \cap B_{\rho}(b), \\ 
\end{aligned}
\right .
\label{extension}
\end{equation}
where $h_1$ minimizes the Dirichlet energy in its domain of definition with respect to the boundary data $(1/2)^\alpha h$ on $C_{\rho} \cap \{ t = 1/2 \}$
and $v$ elsewhere, and likewise $h_2$ minimizes the Dirichlet energy in its domain of definition with respect to the boundary data $(3/2)^\alpha h$ on $C_{\rho} \cap \{ t = 3/2 \}$
and $v$ elsewhere.
 
We compute:
\begin{equation*}
\begin{split}
\int_{C_\rho} | \nabla \tilde{h} |^2 \, \mathrm{d}x = 
\int_{C_\rho \cap \left \{ t \in [-\rho + 1/2, 1/2] \right \} } | \nabla h_1 |^2 \, \mathrm{d}x +
\int_{C_\rho \cap \left \{ t \in [3/2, 3/2 + \rho] \right \} } | \nabla h_2 |^2 \, \mathrm{d}x \; + \\
\int_{1/2}^{3/2} \alpha^2 t^{ 2 \alpha} \left ( \int_{\mathbf{S}^2 \cap B_{\rho}(b)} | h |^2 \, \mathrm{d}\mathcal{H}^2 \right ) \, \mathrm{d}t +
\int_{1/2}^{3/2} t^{2\alpha}  \left ( \int_{\mathbf{S}^2 \cap B_{\rho}(b)} \left | \nabla_{\mathrm{tan} } h  \right |^2 \, \mathrm{d}\mathcal{H}^2 \right ) \, \mathrm{d}t. \\
\end{split}
\end{equation*}
If \eqref{almostminimality} fails, then we have:
\begin{equation}
\begin{split}
\int_{C_\rho} | \nabla \tilde{h} |^2 \, \mathrm{d}x & < 
\int_{C_\rho} | \nabla v |^2 \, \mathrm{d}x + 
\int_{C_\rho \cap \left \{ t \in [-\rho + 1/2, 1/2] \right \} } | \nabla h_1 |^2 \, \mathrm{d}x \\
& + \int_{C_\rho \cap \left \{ t \in [3/2, 3/2 + \rho] \right \} } | \nabla h_2 |^2 \, \mathrm{d}x + 
A_1(\alpha) M_0 \rho
- C A_2 (\alpha) \rho.
\end{split}
\end{equation}
Since there exists a bi-Lipschitz homeomorphism from $\partial C_\rho$ to $S^2$, which extends to a bi-Lipschitz map from $C_\rho$ to $B^3_1(0)$, 
following the argument in the proof of \cite[Lemma 4.1]{SU82}, we obtain:
\begin{equation*}
\begin{split}
\int_{C_\rho \cap \left \{ t \in [-\rho + 1/2, 1/2] \right \} } | \nabla h_1 |^2 \, \mathrm{d} \mathcal{H}^2
\leq \rho  \int_{(-\rho + 1/2) \left ( S^2 \cap B_\rho(b) \right ) } | \nabla_{\mathrm{tan}} v |^2 \, \mathrm{d}x \; + \\
\rho \int_{(1/2) \left ( S^2 \cap B_\rho(b) \right)} | \nabla_{\mathrm{tan}} h_1 |^2 \, \mathrm{d}\mathcal{H}^2 + 
\rho \int_{-\rho + 1/2}^{1/2} \int_{t \left ( S^2 \cap\partial  B_\rho(b) \right )} | \nabla_{\mathrm{tan}} v |^2 \, \mathrm{d}\mathcal{H}^1 \, \mathrm{d}t, 
\end{split}
\end{equation*}
and consequently:
\begin{equation*}
\int_{C_\rho \cap \left \{ t \in [-\rho + 1/2, 1/2] \right \} } | \nabla h_1 |^2 \, \mathrm{d} \mathcal{H}^2 < A_3 \left ( E_0  \right ) \rho,
\end{equation*}
where we note that $A_3$ depends on $E_0$ linearly.
Clearly, the same bound holds for $h_2$ as well. Thus, we obtain:
\begin{equation}
\int_{C_\rho} | \nabla \tilde{h} |^2 \, \mathrm{d}x < \int_{C_\rho} | \nabla v |^2 \, \mathrm{d}x + 2 A_3 \left ( E_0  \right ) \rho + A_1(\alpha) M_0 \rho
- C A_2 (\alpha) \rho.
\label{contramin}
\end{equation}
Hence choosing $C$ large enough with respect to $\alpha$ and $E_0$ in \eqref{contramin}, we get:
\begin{equation*}
\int_{C_\rho} | \nabla \tilde{h} |^2 \, \mathrm{d}x < 
\int_{C_\rho} | \nabla v |^2 \, \mathrm{d}x,  
\end{equation*}
which contradicts the minimality of $v$ on the conical domain $C_\rho$ and which establishes the claim. 

Finally note that $A_3 \left (E_0 \right )$ and $M_0$ would be replaced with $\lambda^2 A_3 \left (E_0 \right ) $ and $\lambda^2 M_0$ respectively, 
if we replace $v$ with $\lambda v$. Hence, replacing $C$ with $\lambda^2 C$ would give the estimate \eqref{almostminimality} for $\lambda v$.
\end{proof}

Finally, we discuss the zero sets of homogenenous minimizers on $\mathbf{R}^3$ in the spirit of \cite[Theorem 3.1]{HardtLin90}.

\begin{lemma} \label{3DHomogeneous}
There exist positive constants $N_0$, $d_0$, $C$ so that for any non-constant homogeneous energy-minimizing map $v: \mathbf{R}^3 \to \mathbf{D}_k$,
$v^{-1}\{ 0 \} \cap \mathbf{S}^2$ consists of $2m$ points separated by distances at least $d_0$, where $2m \leq K_0$,
where both $d_0$ and $K_0$ depend on $N_v \left ( 0; 0^+ \right )$. Near each $a \in v^{-1}\{ 0 \} \cap \mathbf{S}^2$, the following asymptotic estimate holds:
\begin{equation}
\left | v(x) - w_a \circ p_a (x-a) \right | \leq C \left ( |x-a|^{1/2 \sqrt{k}} \right ),
\label{asymptoticEstimate} 
\end{equation}
for some two-dimensional minimizer $w_a$ and orthogonal projection $p_a \, : \, \mathbf{R}^3 \to \mathbf{R}^2$ with $p_a(a) = 0$. 
Furthermore, for $\epsilon > 0$,
there exist $\beta = \beta \left ( \epsilon, N \left (0; 0^+ \right ) \right )> 0$ and $\gamma = \gamma \left ( \epsilon, N \left ( 0; 0^+ \right ) \right ) > 0$ such that:
\begin{equation}
N_v \left ( b ; r \right ) \leq \frac{1}{2 \sqrt{k} } + \epsilon,
\label{frequencyperturbation}
\end{equation}
whenever $ \left | a - \frac{b}{|b|} \right | < \beta$ and
$r \in (0, |b| \gamma ]$ for some $a \in v^{-1} \{0\} \cap \mathbf{S}^2$.  
\end{lemma}

\begin{proof}
For $a \in v^{-1}\{ 0 \} \cap \mathbf{S}^2$, consider a tangent map $v_0$ at $a$. Then there exists a blow-up sequence $v_{\lambda_i}$ at $a$, 
defined as in Lemma \ref{existenceHblowups}, such that $v_0$ is the limit of $v_{\lambda_i}$, and it is a homogeneous energy-minimizing map by Lemmas \ref{compactness} 
and \ref{existenceHblowups}.

The homogeneity of $v$ and $v(a) = 0$ together imply that each $v_{\lambda_i}(x) = 0$, whenever 
$x \in \mathrm{span} \left ( \{ a \} \right ) \cap B_{1/ \lambda_i}(0)$. Hence, by the local uniformity of the convergence, we have
$v_0 \equiv 0$ on $\mathrm{span} \left ( \{ a \} \right )$. Furthermore, choosing Euclidean coordinates centered at $a$ so that $x^1$ is radial
at $a$ we have $v_0$ independent of $x^1$. To see this fact, note that for $x,y \in \mathbf{R}^3$ such that $x-y = \mu a$, we have:
\begin{equation*}
\begin{aligned}
v_0 (x) = \lim_{\lambda_i \downarrow 0} \frac{ v \left ( a + \lambda_i x \right ) }{\left ( \fint_{B_{\lambda_i}(a)} |v|^2 \, \mathrm{d}x \right )^{1/2}}
& = \lim_{\lambda_i \downarrow 0} \frac{ v \left ( (1+ \lambda_i \mu)a + \lambda_i y \right ) }{\left ( \fint_{B_{\lambda_i}(a)} |v|^2 \, \mathrm{d}x \right )^{1/2}} \\
& = \lim_{\lambda_i \downarrow 0} \left [ \left (1+ \lambda_i \mu \right )^{\alpha} \frac{ v \left ( a + \frac{\lambda_i}{1+ \lambda_i \mu} y \right ) }
{\left ( \fint_{B_{\lambda_i}(a)} |v|^2 \, \mathrm{d}x \right )^{1/2}} \right ] \\
& = \lim_{\lambda_i \downarrow 0} \left [ (1+ \lambda_i \mu)^{\alpha}  v_{\lambda_i} \left ( \frac{1}{1 + \lambda_i \mu } y \right ) \right ] = v_0(y).
\end{aligned}
\end{equation*}
Here the third equality is due to the $\alpha$-homogeneity of $v$ for some $\alpha > 0$, as $\left ( 1+ \lambda_i \mu \right ) > 0$ for $i$ large enough.
The last equality follows from the locally uniform convergence and the uniform H\"older estimate in Lemma \ref{holderestimate}. 
Thus, we conclude that any tangent map of $v$ at $a$ is of the form $w_a \circ p_a$, for some
$w_a$, a homogeneneous energy-minimizer defined on $\mathbf{R}^2$ and orthogonal projection $p_a : \mathbf{R}^3 \to \mathbf{R}^2$ with $p_a(a) = 0$.

By the triangle inequality:
\begin{equation*}
\left | v(x) - w_a \circ p_a (x-a) \right | \leq |x-a|^{1/2\sqrt{k}} \left ( \frac{ \left | v(x) \right | }{|x-a|^{1/2\sqrt{k}}} + h_a \left ( \frac{x-a}{|x-a|} \right ) \right ),  
\end{equation*}
where $w_a \circ p_a (z) = |z|^{1/2\sqrt{k}} h \left ( z / |z| \right )$. Since $|h| \leq M$, it suffices to show that in a neighborhood of $a$:
\begin{equation*}
\sup_{x \neq a} \frac{ \left | v(x) \right | }{|x-a|^{1/2\sqrt{k}}} < C.
\end{equation*}
If this claim is false, then there exist $x_i = a + r_i \omega_i$ for $\left | \omega_i \right | = 1$, $r_i \to 0$, as $i \to \infty$ such that:
\begin{equation}
\lim \sup_{i \to \infty} \frac{ \left | v \left ( a + r_i \omega_i \right ) \right | }{ r_i^{1/2\sqrt{k}}} = \infty.
\label{suboptimalEst}
\end{equation}
Defining:
\begin{equation*}
v_i(x) = \left ( \fint_{B_{r_i}(a)} |v|^2 \right )^{-1/2} v \left ( a + r_i x \right ), 
\end{equation*}
for a subsequence, by \eqref{existenceHblowups}, $v_i$ converges to a minimizer $\tilde{w}$, homogeneous of degree $1/2\sqrt{k}$. 
In particular, $\left | v_i (x) \right | \leq C$ for $i$ large.
Thus, choosing $x = \omega_i$:
\begin{equation*}
\frac{v \left ( a + r_i \omega_i \right )}{r_i^{1/2\sqrt{k}}} \leq C \left ( \frac{1}{r_i^{3+1/\sqrt{k}}} \int_{B_{r_i}(a)} |v|^2 \, \mathrm{d}x  \right )^{1/2}.
\end{equation*}
However, using \eqref{Hcomparison}, \eqref{subharmonicaverage} and \eqref{almgrenmonotonicity}, we estimate for $r_i < 1$:
\begin{equation*}
\frac{1}{r_i^{3+ \frac{1}{\sqrt{k}}}} \int_{B_{r_i}(a)} |v|^2 \, \mathrm{d}x = \frac{1}{r_i^{3+ \frac{1}{\sqrt{k}}}} \int_0^{r_i} H(a;s) \, \mathrm{d}s
\leq \frac{H(a;1)}{r_i^{3+ \frac{1}{\sqrt{k}}}} \int_0^{r_i} s^{2 + \frac{1}{\sqrt{k}}} \, \mathrm{d}s = \frac{H(a;1)}{3 + \frac{1}{\sqrt{k}}}.
\end{equation*}
But then \eqref{suboptimalEst} cannot be true, and \eqref{asymptoticEstimate} is proved. 

To verify that the finite set $A$ of zeroes of $v|_{S^2}$ has even cardinality, we will use the continuous mapping
\begin{equation*}
(P\circ v)|_{{\mathbf S}^2\setminus A}\ :\ {\mathbf S}^2\setminus A\ \to\ \mathbf{RP}^2,
\end{equation*}
where the projection  $P: {\mathbf D}_k\setminus \{0\}\to \mathbf{RP}^2$ is given by $P\left(\rho \left( \sqrt{k-1},[\xi] \right)\right) = [\xi ]$ for $\rho > 0$ and $\xi\in {\mathbf S}^2$.
Recalling  the formula $v_0=w_a\circ p_a$ for the tangent map at any $a\in A$, we see that, 
by restricting to a sufficiently small circle ${\bf S}^2\cap\partial{\bf B}_\varepsilon (a)$, 
the resulting homotopy class $\left[ (P\circ v)|_{{\mathbf S}^2\cap\partial{\mathbf B}_\varepsilon (a)} \right]$ is nonzero in  $\Pi_1(\mathbf{RP}^2)\cong{\mathbf Z}_2$.  
Since the whole sphere ${\mathbf S}^2$ is simply-connected, it is well-known from topology that: 
\begin{equation}
\sum_{a\in A}\left[ (P\circ v)|_{{\mathbf S}^2\cap\partial B_\varepsilon (a)} \right]\ 
=\ 0\in\Pi_1(\mathbf{ RP}^2)\cong{\mathbf Z}_2\ ; \label{HomotopyConstruction}
\end{equation}
hence, the cardinality of $A$ is even.
To verify \eqref{HomotopyConstruction}, one may, for example, find an isotopy of ${\mathbf S}^2$ that moves $A$ to  a finite subset of the equator. 
By joining each small circle ${\mathbf S}^2\cap\partial B_\varepsilon (a)$ to the north pole and back again along a logitudinal arc, 
one constructs a single parameterized loop that is homotopic in ${\mathbf S}^2\setminus A$ to the sum of these small circles. 
The resulting loop may then be homotoped in ${\mathbf S}^2\setminus A$ to the constant south pole map. 
Composing these homotopies with $P\circ v$ then gives \eqref{HomotopyConstruction}.

In order to obtain $d_0$, we argue by compactness. If there exists no such $d_0$, we can find a sequence of homogeneous minimizing-maps
$v_i \, : \mathbf{R}^3 \to \mathbf{D}_k$ with uniformly bounded Almgren frequencies $N_{i} \left ( 0 ; 0^+ \right ) \leq A_0$
and corresponding distinct points $a_i, b_i \in v_i^{-1}\{ 0 \} \cap \mathbf{S}^2$ such that
$2 r_i = \left | a_i - b_i \right | \to 0$, as $i \to \infty$. Composing each map $v_i$ with an appropriate rotation, we may assume
$b_i = (0,1) \in \mathbf{R}^2 \times \mathbf{R}$. 

We define $w_i \, : \, B_1(0) \to \mathbf{D}_k$ as:
\begin{equation*}
 w_i(x) = \left ( \fint_{B_{r_i}\left ( (0,1) \right )} \left | v_i \right |^2 \, \mathrm{d}x \right )^{-1/2} v_i \left ( (0,1) + r_i x \right ) .
\end{equation*}
Once we verify that after passing to a subsequence $w_i$ converges to a homogeneous energy-minimizing map $w$ strongly and uniformly on $B_1(0)$, 
we will obtain that $w(0) = 1$, as well as $ w \left ( x_\infty \right ) = 0$, where $x_i = r_i^{-1} \left (a_i - (0,1) \right ) \to x_\infty$
and $\left | x_\infty \right | = 1/2$ by construction. 
Hence, by the homogeneity of $w$ at $0$, $w$ vanishes on the ray containing $0$ and $x_\infty$. 
On the other hand, by the above argument $w$ is independent of the $x_3$-direction, which is radial at $(0,1)$. 
However, the ray containing $0$ and $x_\infty$ is orthogonal to the $x_3$-direction, since $\left | x_i \right | = 1/2$ and
$\left \langle x_i , (0,1) \right \rangle \to 0$ as $i \to 0$. 
Therefore, $w^{-1} \{0 \}$ has positive $\mathcal{H}^2$-measure, contradicting Lemma \ref{Hdimension}.

By the proof of Lemma \ref{existenceHblowups}, each $w_i$ is bounded in $H^1 \left ( B_1 (0) \right )$ by a constant depending on $N_{v_i} \left ( 0; 0^+ \right )$.
Hence, the uniform bound on the frequencies of $v_i$ at $0$ yields a uniform bound on the $H^1$-norms of $w_i$, and by Rellich's theorem a subsequence of
$w_i$ converge weakly in $H^1$ and strongly in $L^2$ to a map $w$. 
Moreover, $\left \| w_i \right \|_{L^2 \left ( B_1(0) \right )} = 1$ and Lemma \ref{holderestimate} together imply uniform convergence. 

Finally, strong convergence in $H^1 \left ( B_1 (0) \right )$ follows from the Caccioppoli inequality \eqref{caccioppoli}.
In order to see that \eqref{caccioppoli} holds for $w_i$ with a uniform constant $C$, we derive a uniform $L^{\infty}$-bound on $w_i$.
Reexpressing $w_i$, we have:
$$
\sup_{z \in B_2(0)} \left | w_i(z) \right |^2 = \sup_{y \in B_{2r_i}((0,1))} \left | v_i(y) \right |^2 \cdot \left [ \fint_{B_{r_i}((0,1))} \left | v_i \right |^2 \, \mathrm{d}x \right ]^{-1}.
$$
By the subharmonicity of $\left | v_i \right |^2$:
$$
\sup_{z \in B_2(0)} \left | w_i(z) \right |^2 
\leq \sup_{y \in B_{2r_i}((0,1))} \fint_{B_{\delta r_i}(y)} \left | v_i \right |^2 \, \mathrm{d}x \cdot \left [ \fint_{B_{r_i}((0,1))} \left | v_i \right |^2 \, \mathrm{d}x \right ]^{-1}.
$$
By inclusion and \eqref{Hcomparison}:
\begin{equation*}
\begin{aligned}
\sup_{z \in B_2(0)} \left | w_i(z) \right |^2
& \leq \left ( \frac{2r_i + \delta r_i}{\delta r_i} \right )^3 \left ( \frac{  \fint_{B_{2r_i + \delta r_i}((0,1))}
\left | v_i \right |^2 \, \mathrm{d}x }{ \fint_{B_{r_i}((0,1))} \left | v_i \right |^2 \, \mathrm{d}x } \right ) \\
& \leq \delta^{-3} \left ( \frac{2r_i + \delta r_i}{r_i} \right )^{3 + 2 N_i \left ( (0,1) ; 2r_i + \delta r_i \right )} \\
& \leq \delta^{-3} \left ( 2 + \delta \right )^{3 + 2 N_i \left ( (0,1) ; 2r_i + \delta r_i \right )}.
\end{aligned}
\end{equation*}
Finally, the homogeneity of $v_i$, the local frequency estimate \eqref{localFcontrol}, and the uniform bound on $N_{v_i} \left ( 0; 0^+ \right )$ give:
\begin{equation*}
2 N_i \left ( (0,1) ; 2r_i + \delta r_i \right ) \leq C_1 N_i \left ( 0; 0^+ \right ) + C_2 \leq C_1 A_0 + C_2.
\end{equation*}
Hence, we obtain a uniform $L^{\infty}$-bound on $w_i$, which depends on a fixed number $\delta > 0$ and the frequency bound $A_0$. 

Hence, the proof that there exists a minimal distance $d_0$, depending on $N_v \left ( 0 ; 0^+ \right )$ is complete.
Therefore, there also exists an upper bound $K_0 = C/d_0^2$ for the cardinality of $v^{-1} \{ 0 \} \cap \mathbf{S}^2$, for a uniform constant factor $C > 0$.

Finally, to prove \eqref{frequencyperturbation}, it follows from Lemma \ref{existenceHblowups} and \eqref{asymptoticEstimate} that for every
$a \in v^{-1} \{ 0 \} \cap \mathbf{S}^2$, we have $N_v \left ( a; 0^+ \right ) = \frac{1}{2 \sqrt{k}}$. First choose $\delta = \delta(\epsilon) < d_0 /4$ such that by \eqref{almgrenmonotonicity}:
\begin{equation*}
N_v (a ; \gamma ) \leq \frac{1}{2 \sqrt{k} } + \frac{\epsilon}{2}.
\end{equation*}
Then by the upper-semicontinuity of $N_v(\, \cdot \, ; \gamma )$, as proved in \eqref{uppersemicontinuity}, we can choose
$\beta < d_0 / 8$ such that:
\begin{equation*}
N_v \left ( c ; \gamma \right ) < \frac{1}{2 \sqrt{k} } + \epsilon,
\end{equation*}
whenever $ | c - a | < \beta$. 
Letting $c = |b|^{-1} b $, scaling and using the homogeneity of $v$ around $0$, the estimate in \eqref{frequencyperturbation} follows for $\beta$ and $\gamma$ possibly depending on $v$ as well. 

In order to prove that $\beta$ and $\gamma$ can be chosen to depend on $\epsilon$ and $N_v \left ( 0; 0^+ \right )$ only, we argue by contradiction. 
Let $\rho > 0$ be arbitrary. 
If the claim fails, for some $\epsilon_0 > 0$, 
there exist $v_i$ with $N_{v_i} \left ( 0; 0^+ \right ) < M$, $b_i \in \overline{B}_1(0) \backslash B_{\rho} (0)$ and $\gamma_i$ such that after composing each map with a suitable rotation 
$\left | (0,1) - \left | b_i \right |^{-1} b_i \right | \to 0$ and $\gamma \to 0$ as $i \to \infty$, 
while $N_{v_i} \left ( b_i; \gamma_i \left | b_i \right | \right ) > \frac{1}{2 \sqrt{k}} + \epsilon_0$. 
However, we note using the homogeneity of $v_i$ that $\left | b_i \right |^{-1} b_i \in v^{-1} \{ 0 \} \cap \mathbf{S}^2$, 
and likewise we can calculate $N_{v_i} \left ( \left | b_i \right |^{-1} b_i ; \gamma_i \right )$. 
Using the minimality of $d_0$ and the compactness of minimizers, we consider a blow-up sequence at $a = (0,1)$ with blow-up scales $s_i = 2 \left | a - \left | b_i \right |^{-1} b_i \right |$. 
Since $\left | a - c \right | \geq d_0$ for any other $c \in v^{-1} \{ 0 \} \cap \mathbf{S}^2$, the homogeneous blow-up at $a$ is as in \eqref{2Dhminimizer} as before. 
However, the lower bound for $N_{v_i} \left ( \left | b_i \right |^{-1} b_i ; \gamma_i \right )$ gives that the tangent map has a zero on $\partial B_{1/2} (0) \times \{ 0 \} $ as well, 
contradicting Lemma \ref{Hdimension}.
\end{proof}

\begin{remark}
Arguing as in \cite{Simon83}, the asymptotic estimate \eqref{asymptoticEstimate} can be improved to:
\begin{equation*}
\left | v(x) - w_a \circ p_a (x-a) \right | = o \left ( |x-a|^{1/2 \sqrt{k}} \right ) .
\end{equation*}
In particular, the tangent map is unique in this special case. 
The proof involves showing that the Jacobi fields on $S^1$ along a two-dimensional homogeneous minimizer as above are integrable. 
Combining this observation with estimates derived from \eqref{almgrenmonotonicity}, an expansion in terms of the eigenfunctions of the Jacobi operator on $S^1$ and separation of variables,
we can obtain the sharp estimate. See \cite{AUniqueness} for the details. 
\end{remark}

\begin{corollary} \label{charhom}
For any non-constant homogeneous energy-minimizing map $v: \mathbf{R}^3 \to \mathbf{D}_k$, the set $v^{-1}\{0\}$ consists of $2m$ incoming rays
joining at $0$, where $2m \leq K_0 = K_0 \left ( N_v \left (0; 0^+ \right ) \right )$. If $m \geq 1$, 
then $N_v \left (0; 0^+ \right ) \geq 1/ \left ( 2 \sqrt{k} \right )$, where the equality holds 
if and only if $v = w_a \circ p_a$ for a homogeneous minimizer $w_a \, : \, \mathbf{R}^2 \to \mathbf{D}_k$ as in \eqref{2Dhminimizer} 
and an orthogonal projection $p_a \, : \, \mathbf{R}^3 \to \mathbf{R}^2$ with $p_a(a) = 0$.
\end{corollary}

\begin{proof}
Note that Lemma \ref{3DHomogeneous} and the homogeneity of $v$ together determine the structure of $v^{-1}\{0\}$.
The lower bound for $N_v \left (0; 0^+ \right )$ follows from the upper semicontinuity of $N \left ( x; 0^+ \right )$ at every $x \in \mathbf{R}^3$ and
the classification of tangent maps at $a \not = 0$ in Lemmas \ref{2Dhminimizer} and \ref{3DHomogeneous}.
If $v$ depends on two variables, then the equality holds by Lemma \ref{2Dhminimizer}. 
Conversely, by Lemma \ref{monotonocity} and \eqref{exactLocalFcontrol}, the equality $N_v \left (0; 0^+ \right ) = 1/ \left ( 2 \sqrt{k} \right )$ implies
\begin{equation*}
N_v \left ( a ; 0^+ \right ) = N_v \left ( a; \rho \right ), 
\end{equation*}
for every $\rho > 0$, which is equivalent to $v$ being homogeneous at $a$ by Lemma \ref{monotonocity}. We observe that if $v$ is homogeneous
at both $0$ and $a$, then $v$ is independent of $x_3$-variable determined by the ray from $0$ to $a$ and vanishes on $\mathrm{span}\left ( \{ a \} \right )$.
Hence, we have $v = w_a \circ p_a$, where $w_a$ is determined by the formula \eqref{2Dhminimizer} and  $p_a \, : \, \mathbf{R}^3 \to \mathbf{R}^2$ 
is an orthogonal projection with $p_a(a) = 0$.
\end{proof}

\section{Isolated and Non-isolated Defects} \label{Decomposition}

In this section we introduce a decomposition for the defect set $u^{-1} \{ 0 \}$ of an arbitrary, non-constant map $u \, : \, \Omega \to \mathbf{D}_k$
minimizing the modified Ericksen energy.

\begin{definition}
We denote $u^{-1} \{ 0 \} = \mathcal{Z}_0 \cup \mathcal{Z}_1$, where:
\begin{equation}
\mathcal{Z}_0 = \left \{ a \in u^{-1} \{ 0 \} \, : \, u_{\infty}^{-1} \{ 0 \} \cap \mathbf{S}^2 = \emptyset \; \mathrm{for} \; \mathrm{every} \; \mathrm{tangent} \; \mathrm{map}
\; u_{\infty} \; \mathrm{at} \; a \right \},
\end{equation}
and
$\mathcal{Z}_1 = u^{-1} \{ 0 \} \backslash \mathcal{Z}_0$. 
\end{definition}

The main result of this section is that $\mathcal{Z}_1$ can be approximated at each small scale by the defect set of a corresponding homogeneous minimizer as analyzed
in Lemma \ref{3DHomogeneous}, whereas $\mathcal{Z}_0$ consists of isolated points. Firstly,
we prove a lemma that addresses the case where two-dimensional local behavior at a zero of the map $u$, captured by the Almgren frequency as in Corollary \ref{charhom}, 
is perturbed slightly. 

\begin{lemma} \label{nonisolated1}
For every $\epsilon > 0$, there exists a $\delta = \delta ( \epsilon )$ so that if $u \, : \, B_2(0) \to \mathbf{D}_k$ is energy minimizing,
$\left \| u \right \|_{L^\infty \left ( \partial B_2(0) \right ) } \leq M$,
$\left ( \overline{B}_1(0) \backslash B_{\frac{1}{2}}(0) \right ) \cap u^{-1} \{0\} \neq \emptyset$, and
\begin{equation*}
N_u \left ( 0; 2 \right ) < \frac{1}{2 \sqrt{k}} + \delta,
\end{equation*}
then for some $w \, : \, \mathbf{R}^2 \to \mathbf{D}_k$, a homogeneous minimizer on $\mathbf{R}^2$, and some orthogonal projection $p \, : \, \mathbf{R}^3 \to \mathbf{R}^2$,
the following hold:
\begin{equation}
\left \| u - w \circ p \right \|_{H^1 \left ( B_2(0) \right )} < \epsilon, \label{nonisolated1Ea}
\end{equation}
\begin{equation}
\overline{B_1(0)} \cap u^{-1} \{0\} \subset \left \{ x \, : \, \mathrm{dist} \left ( x, p^{-1} \{0\} \right ) < \epsilon \right \},  \label{nonisolated1Eb}
\end{equation}
and for every $z \in [-1,1]$, $ \left ( B_{\epsilon}^3(0) \times \{ z \} \right ) \cap \mathcal{Z}_1 \neq \emptyset$; hence,
\begin{equation}
\overline{B_1(0)} \cap p^{-1} \{ 0 \} \subset \left \{ x \, : \, \mathrm{dist} \left ( x, \mathcal{Z}_1 \right ) < \epsilon \right \}. \label{nonisolated1Ec}
\end{equation}
\end{lemma}

\begin{proof}
Suppose there is an $\epsilon > 0$ such that for each positive integer $i$, there exists a minimizing map $u_i \; B_2(0) \to \mathbf{D}_k$, such that
 $\left \| u_i \right \|_{L^\infty \left ( \partial B_2(0) \right ) } \leq M$, $v^{-1}\{0\} \cap \left ( \overline{B_1(0)} \backslash B_{\frac{1}{2}}(0) \right ) \neq \emptyset$,
and
\begin{equation*}
N_i \left ( 0; 2 \right ) = N_{u_i} \left ( 0; 2 \right ) < \frac{1}{2 \sqrt{k}} + \frac{1}{i},
\end{equation*}
yet one of \eqref{nonisolated1Ea}, \eqref{nonisolated1Eb} or \eqref{nonisolated1Ec} fails for $u_i$.

Note that the bound $\left \| u_i \right \|_{L^\infty \left ( \partial B_2(0) \right ) } \leq M$ and Lemma \ref{maximumprinciple} give a uniform $L^{\infty}$-bound
for $u_i$ in $\overline{B_2(0)}$. Moreover, combining it with the uniform bound for $N_i \left ( 0; 2 \right )$ gives a uniform $H^1$-bound for $u_i$ in $B_2(0)$.
Applying Lemmas \ref{caccioppolitype} and \ref{compactness} with these uniform bounds yields strong convergence in $H^1$ for a subsequence of $u_i$. 
Hence, we may assume that $u_i \to v$ strongly in $H^1 \left ( B_2(0) \right )$, and by the trace theorem strongly in $L^2 \left ( \partial B_2(0) \right )$ as well. 
Consequently, we have:
\begin{equation*}
N_v \left ( 0; 2 \right ) \leq \liminf_{i \to \infty } N_{u_i} \left ( 0; 2 \right ) \leq \frac{1}{2 \sqrt{k}}.
\end{equation*}
However, also taking into account the monotonicity of the Almgren frequency \eqref{almgrenmonotonicity}, we have:
\begin{equation*}
\begin{split}
N_v \left ( 0; 2 \right ) \geq N_v \left ( 0; 0^+ \right ) 
= \lim_{r \to 0 } N_v \left ( 0; r \right ) 
= \lim_{r \to 0 } \lim_{i \to \infty} N_{u_i} \left ( 0; r \right ) \geq 
\\ \lim_{r \to 0 } \lim_{i \to \infty} N_{u_i} \left ( 0; 0^+ \right )
= \lim_{i \to \infty} N_{u_i} \left ( 0; 0^+ \right ) 
\geq \lim_{i \to \infty} \frac{1}{2 \sqrt{k}}
= \frac{1}{2 \sqrt{k}}.
\end{split}
\end{equation*}
In addition, since $u_i$ are uniformly bounded in $L^2$
and $L^\infty$ in $B_2(0)$, by Lemma \ref{holderestimate}, after passing to a subsequence if necessary, we can assume that they converge locally uniformly.
Since there exist $a_i \in u_{i}^{-1}\{0\} \cap \left ( \overline{B_1(0)} \backslash B_{\frac{1}{2}}(0) \right )$, we can also assume that 
$a_i$ converge to $a_\infty \in \overline{B_1(0)} \backslash B_{\frac{1}{2}}(0)$. 
By the locally uniform convergence of $u_i$ to $v$ in $B_2(0)$ and Lemma \ref{holderestimate} applied to $v$, which is a minimizer by
Lemma \ref{compactness}, $v \left ( a_\infty \right ) = 0$. 
Thus, by Corollary \ref{charhom}, we have $v = w \circ p$ for a suitable projection $p \, : \, \mathbf{R}^3 \to \mathbf{R}^2$
such that $p(a) = 0$ and $w : B_2(0) \to \mathbf{D}_k$ energy minimizing. In particular, \eqref{nonisolated1Ea} holds for $i$ large enough.
Choosing our coordinate system accordingly, we can assume $p(y,z) = y$ for $(y,z) \in \mathbf{R}^2 \times \mathbf{R}$.

Since $v$ is H\"older continuous on $B_2(0)$, we can define $c_0$ as its minimum on the compact domain 
$\overline{B_{\frac{3}{2}}(0)} \backslash \left ( B^2_\epsilon (0) \times \mathbf{R} \right )$.
Then by the locally uniform convergence of $u_i$ to $v$ in $B_2(0)$ and the H\"older continuity of $u_i$, for $i$ large enough, $u_i \geq c_0/2$ on
 $\overline{B_{\frac{3}{2}}(0)} \backslash \left ( B^2_\epsilon (0) \times \mathbf{R} \right )$. Consequently, \eqref{nonisolated1Eb} holds for $u_i$,
for $i$ large enough.

Finally, by the locally uniform convergence of $u_i$ to $v$ on $B_2(0)$, for all $z \in [-1,1]$ and for $i$ large enough,
the homotopy class
\begin{equation}
\left[ (P\circ u_i)|_{\partial B^2_\varepsilon(0)\times \{ z\} } \right]\ =\ \left[ (P\circ v)|_{\partial B^2_\varepsilon(0)
\times \{ z\} } \right]\ =\ 1\in \Pi_1 \left (\mathbf{RP}^2 \right )\cong{\mathbf Z}_2,
\label{HomotopyClass}
\end{equation}
as $v = w \circ p$ and possible $w$ are classified in Lemma \ref{2Dhminimizer}. 
Hence, by the argument of Remark \ref{linedefect}, for each $z \in [-1,1]$, $u_i$ must vanish inside $B^2_{\epsilon}(0) \times \{ z \}$ for $i$ large enough. 
In particular, for $i$ large enough:
\begin{equation*}
\overline{B_1(0)} \cap p^{-1} \{ 0 \} \subset \left \{ x \, : \, \mathrm{dist} \left ( x, u_i^{-1} \{0\} \right ) < \epsilon \right \}.
\end{equation*}

Now suppose  
$\left ( B^2_{\epsilon}(0) \times \{ z \} \right ) \cap 
u_i^{-1} \{ 0 \} \subset \mathcal{Z}_0^i$.
Then $\left ( B^2_{\epsilon}(0) \times \{ z \} \right ) \cap 
u_i^{-1} \{ 0 \}$ is a finite set, 
since otherwise a blow-up argument
as in Lemma \ref{3DHomogeneous} would give a
contradiction. We denote this set of zeros as
$\left \{ a_1, a_2, \, ... \, , a_j \right \} $.
Then choosing $\sigma > 0$ small enough,
$u_i$ would be non-zero on the modification of
$B^2_{\epsilon}(0) \times \{ z \} $ obtained
by replacing each flat disk
$B^2_{\sigma} \left ( a_i \right ) \times \{ z \}$ with
the upper hemisphere
 $\partial^+ B^2_{\sigma} \left ( a_i \right )$, which
 is still a topological disk.
Since $u_i$ does not vanish on this topological disk, we have a contradiction with \eqref{HomotopyClass}. Hence, we conclude:
\begin{equation*}
\left ( B^2_{\epsilon}(0) \times \{ z \} \right ) \cap 
\mathcal{Z}^i_1 \neq \emptyset,
\end{equation*}
and consequently:
\begin{equation*}
\overline{B_1(0)} \cap p^{-1} \{ 0 \} \subset \left \{ x \, : \, \mathrm{dist} \left ( x, \mathcal{Z}^i_1 \right ) < \epsilon \right \},
\end{equation*}
when $i$ is sufficiently large.

In conclusion, for $i$ large enough $u_i$ satisfies all three
of \eqref{nonisolated1Ea}, \eqref{nonisolated1Eb},
and \eqref{nonisolated1Ec}, and the lemma is proved by
this contradiction of the initial claim. 
\end{proof}

Secondly, we study the local behavior at an arbitrary
point in $\mathcal{Z}_1$ in a small enough annular
region around it.

\begin{lemma} \label{nonisolated2}
For any energy minimizing map
$u \, : \, B_1(0) \to \mathbf{D}_k$ 
with $0 \in \mathcal{Z}_1$ 
and for every $\epsilon > 0$, there exists an
$R = R \left ( \epsilon, N \left ( 0; 0^+ \right ) \right ) > 0$ and
positive $2m \leq K_0 = 
K_0 \left ( N \left ( 0; 0^+ \right ) \right ) $, so that
for each $r \in (0, R ]$ there is a corresponding
homogeneous energy minimizing map $v_r$ such that
$v_r^{-1} \{ 0  \} \cap \mathbf{S}^2$ has
exactly $2m$ points, and for
\begin{equation*}
u_r(x) = \left ( \fint_{ B_r(0)} |u|^2 \, \mathrm{d}A \right )^{-1/2} u(rx),
\end{equation*}
there holds:
\begin{equation}
\left \| u_r - v_r \right \|_{H^1 \left ( B_1(0) \right ) }  \leq \epsilon.
\label{nonisolatedE2a}
\end{equation}
Moreover, the following inclusions hold:
\begin{equation}
\left ( \overline{B_r(0)} \backslash B_{r/2}(0) \right ) 
\cap u^{-1} \{0\} \subset
\left \{ x \; : \; \mathrm{dist} \left (x, v_r^{-1} \{ 0 \} \right )
< r \epsilon \right \},
\label{nonisolatedE2b}
\end{equation}
and
\begin{equation}
\left ( \overline{B_r(0)} \backslash B_{r/2}(0) \right )
\cap v_r^{-1} \{0\} \subset
\left \{ x \; : \; \mathrm{dist} \left (x, \mathcal{Z}_1 \right )
< r \epsilon \right \}.
\label{nonisolatedE2c}
\end{equation} 
\end{lemma}
\begin{proof}
Once we establish the inclusions \eqref{nonisolatedE2b} and
\eqref{nonisolatedE2c} for all sufficiently small $r$ and
corresponding homogeneous energy minimizing map $v_r$,
whenever $\epsilon < d_0 / 8$, then it follows from Lemma \ref{3DHomogeneous} that the cardinality of $v_r^{-1} \{0\} \cap \mathbf{S}^2$ is uniquely determined by that of
$u^{-1} \{0\} \cap \partial B_r(0) $ and is independent
of $r$.

Suppose for contradiction that there is a an energy minimizing map
$u \, : \, B_1(0) \to \mathbf{D}_k$ and a positive $\epsilon < d_0 / 8$ so that one of \eqref{nonisolatedE2a}, \eqref{nonisolatedE2b}
or \eqref{nonisolatedE2c} fails for some sequence $r_i \to 0$ and
any choice of homogeneous energy minimizing maps $v_{r_i}$.
Denoting $u_i = u_{r_i}$, we can assume by Lemma \ref{existenceHblowups} that for $i$ large enough:
\begin{equation*}
\left \| u_{i} - v \right \|_{H^1 \left ( B_1(0) \right ) } < \epsilon,
\end{equation*}
for some homogenous energy minimizing map 
$v \, : \, \mathbf{R}^3 \to \mathbf{D}_k$ with 
$\| v \|_{L^2 \left ( B_1(0) \right )} = 1$, and 
$v^{-1} \{0\} \cap \mathbf{S}^2 \neq \emptyset$, as
$0 \in \mathcal{Z}_1$. 

Note that $v$ attains a minimum $c_0 > 0$, clearly depending on $\epsilon$, on the compact set
$\left \{ x \, : \, \mathrm{dist} \left (x, v^{-1}\{0\} \right ) \geq \epsilon
\right \} \cap \left ( \overline{B_1(0)} \backslash B_{1/2}(0) \right )$,
and since $u_i$ converge to $v$ locally uniformly, $u_i \geq c_0 / 2$ on this compact set for $i$ large enough. Therefore:
\begin{equation*}
\left ( \overline{B_1(0)} \backslash B_{1/2}(0) \right ) \cap u_i^{-1}\{0\}
\subset \left \{ x \, : \, \mathrm{dist}\left (x, v^{-1}\{0\} \right ) < \epsilon \right \},
\end{equation*}
and by scaling:
\begin{equation*}
\left ( \overline{B_{r_i}(0)} \backslash B_{r_i/2}(0) \right ) \cap u_i^{-1}\{0\}
\subset \left \{ x \, : \, \mathrm{dist}\left (x, v_r^{-1}\{0\} \right ) < r_i \epsilon \right \},
\end{equation*}
for $i$ large enough.

The uniform convergence of $u_i$ to $v$ on the compact set
$$ \left \{ x \, : \, \mathrm{dist} \left (x, v^{-1}\{0\} \right ) \geq \epsilon
\right \} \cap \left ( \overline{B_1(0)} \backslash B_{1/2}(0) \right )$$
also implies that by Lemma \ref{3DHomogeneous}, for $i$ large
enough, for all $s \in [ 1/2, 1 ]$ and $a \in v^{-1}\{0\} \cap \mathbf{S}^2$, the two homotopy classes
\begin{equation*}
\left[ (P\circ u_i)|_{\partial B_s(0)\cap\partial B_{\beta s}(sa)} \right]\ =\ \left[ (P\circ v)|_{\partial B_s(0)\cap\partial
 B_{\beta s}(sa)}\right]\ =\ 1\in \Pi_1 \left (\mathbf{RP}^2 \right )\cong\mathbf{Z}_2,
\end{equation*}
for sufficiently small $\beta > 0$.
Arguing as in the proof of Lemma \ref{nonisolated1}, we observe
that for all such $i$, $s$ and $a$, we have a point
$b_i^a \in u_i^{-1} \{0\}$ with $\left | a - \left | b_i^a \right |^{-1} b_i^a \right | < \epsilon$, and furthermore, $b_i^a \in \mathcal{Z}_1^i$
by a similar topological disk construction.
Thus, we obtain:
\begin{equation*}
\partial B_s(0) \cap \partial B_{\epsilon} (s \, a) \cap \mathcal{Z}^i_1
\neq \emptyset,
\end{equation*}
and hence:
\begin{equation*}
\left ( \overline{B_s(0)} \backslash B_{s/2}(0) \right )
\cap v^{-1} \{ 0 \} \subset
\left \{
x \; : \; \mathrm{dist} \, \left ( x , \mathcal{Z}^i_1 \right ) < s \epsilon
\right \},
\end{equation*}
for all $s \in [1/2,1]$ and $i$ large enough.
In conclusion, letting $v_{r_i} = v$, we obtain a
contradiction, as
\eqref{nonisolatedE2a}, \eqref{nonisolatedE2b} and
\eqref{nonisolatedE2c} hold for $v_{r_i}$ and $u_i$ for
$i$ large enough. 
\end{proof}

We are ready to prove that $\mathcal{Z}_0$ is a discrete set.

\begin{corollary} \label{discretezeros}
For $u \, : \, \Omega \to \mathbf{D}_k$, an energy minimizing map, where $\Omega \subset \mathbf{R}^3$, the set $\mathcal{Z}_0$ is
a discrete subset of $\Omega$.
\end{corollary}

\begin{proof}
Lemma \ref{holderestimate}, implies that $u^{-1} \{ 0 \}$ is a relatively closed subset of $\Omega$.
But we observe that if any sequence of points $\left \{ b_i \right \} \subset \mathcal{Z}_0$ has a limit point in $\Omega$, such a point cannot be in $\mathcal{Z}_0$. 
This follows immediately from the definition of $\mathcal{Z}_0$ and the blow-up argument used in establishing the existence of  minimal separation $d_0$ in  Lemma \ref{3DHomogeneous}. 
Hence, any such limit point must lie in $\mathcal{Z}_1$.

Suppose there exists a sequence of points $\left \{ b_i \right \} \subset \mathcal{Z}_0$ with a limit point in $\Omega$. Without loss of generality, we may assume that $b_i \to 0 \in \mathcal{Z}_1$ as
$i \to \infty$. Letting $r_i = (4/3) \left | b_i \right |$ so that
$b_i \in B_{r_i}(0) \backslash B_{r_i / 2}(0)$, by \eqref{nonisolatedE2b}, given any $\epsilon >0$, for $i$ large enough $b_i $ lies in the $r_i \epsilon$-neighborhood
of the zero set of some homogeneous energy minimizing map $v_{r_i}$. Therefore, by Lemma \ref{3DHomogeneous} there exists a point $a_i \in v_{r_i}^{-1} \{ 0 \} \cap \mathbf{S}^2$, 
which is the closest point to $\left | b_i \right |^{-1} b_i$ in this finite set, and
$\left | a_i - \left | b_i \right |^{-1} b_i \right | \to 0$ as $i \to \infty$.

Arguing once again as in the first paragraph, we observe that $\mathcal{Z}_1$ is a closed set. Hence, for each $b_i$, there exists a $c_i$, which is the closest point to $b_i$ in $\mathcal{Z}_1$.
Moreover, \eqref{nonisolatedE2b} and \eqref{nonisolatedE2c}
together imply that $\left | b_i - c_i \right | / r_i \to 0$ as $i \to \infty$.
Fix $\delta = \delta(1/10)$ in Lemma \ref{nonisolated1}. For $i$ large enough, we have by \eqref{nonisolatedE2b} and \eqref{nonisolatedE2c}: 
$ \left | b_i - \left | b_i \right | a_i \right | / r_i <
\beta (\delta/2)$ and $ 2 \left | b_i - c_i \right | / r_i <  \gamma ( \delta/2)$
for $\beta$ and $\gamma$ as in Lemma \ref{3DHomogeneous}.
For such $i$:
\begin{equation*}
N_{v_{r_i}} \left ( b_i ; 2 \frac{ \left | b_i - c_i \right |}{r_i} \right ) < \frac{1}{2 \sqrt{k}} + \frac{\delta}{2}.
\end{equation*}
Since $ \left ( u_{r_i} - v_{r_i}  \right ) \to 0 $ in $H^1 \left ( B_1(0) \right )$, for $i$ large enough we have:
\begin{equation*}
N_{u_{r_i}} \left ( b_i ; 2 \frac{ \left | b_i - c_i \right |}{r_i} \right ) < \frac{1}{2 \sqrt{k}} + \delta,
\end{equation*}
and consequently:
\begin{equation*}
N_{u} \left ( b_i ; 2 \left | b_i - c_i \right | \right ) < \frac{1}{2 \sqrt{k}} + \delta.
\end{equation*}
Letting $s_i = \left | b_i - c_i \right |$, and
\begin{equation*}
w_i(x) = \left ( \fint_{B_{2s_i} \left ( b_i \right )} |u|^2 \, \mathrm{d}A \right )^{-1/2} u \left ( b_i + s_i x \right ),
\end{equation*}
we observe that $d_i = \frac{1}{s_i} \left ( b_i - c_i \right ) \in w_i^{-1} \{ 0 \} \cap \mathbf{S}^2$, and:
\begin{equation*}
N_{w_i} ( 0; 2 ) = N_{u} \left ( b_i ; 2 s_i \right ) < \frac{1}{2 \sqrt{k}}  + \delta.
\end{equation*}
Therefore, applying Lemma \ref{nonisolated1} we obtain for $i$ large enough:
\begin{equation*}
\overline{B_1(0)} \cap w_i^{-1} \{ 0 \} \subset
\left \{ 
x \, : \, \mathrm{dist} \, \left ( x , p^{-1} \{ 0 \} \right ) < 0.1
\right \},
\end{equation*}
\begin{equation*}
\overline{B_1(0)} \cap p^{-1} \{ 0 \} \subset
\left \{
x \, : \, \mathrm{dist} \, \left (x, \mathcal{Z}^i_1 \right ) < 0.1
\right \},
\end{equation*}
for some projection $p \, : \, \mathbf{R}^3 \to \mathbf{R}^2$. 
The first inclusion gives the estimate $\left | d_i \right | = \left | p \left ( d_i \right ) \right | < 0.1$. We choose $e_i$ on the line $p^{-1} \{ 0 \}$ with $ \left | e_i \right | = 1/2$ and $e_i \cdot d_i  > 0$. By the second inclusion there exists an $x_i \in B_{0.1} \left ( e_i \right ) \cap \mathcal{Z}^i_1$. Therefore:
\begin{equation*}
\left | x_i - d_i \right | \leq \left | x_i - e_i \right | + \left | e_i  - d_i \right | < 0.1 + \sqrt{ 0.25 + 0.01 } < 1,
\end{equation*}
and hence:
\begin{equation*}
\left | \left ( s_i x_i +  c_i \right ) - b_i \right | = 
s_i \left | x_i - d_i \right | < s_i,
\end{equation*} 
contradicting the definition of $c_i$. 
\end{proof}

We conclude this section by proving the following corollary, which can be of interest on its own:

\begin{corollary} \label{discretejunctions}
For any $u \, : \, \Omega \to \mathbf{D}_k$, an energy minimizing map, for every compact $K \subset \Omega$ and every $\delta_0$, the set of points $b \in \mathcal{Z}_1 \cap K$
satisfying the frequency lower bound
\begin{equation}
N_u \left ( b; 0^+ \right ) \geq \frac{1}{2 \sqrt{k}} + \delta_0 \label{lowerbound}
\end{equation}
is a discrete set.
\end{corollary}

\begin{proof}
Suppose there exist infinitely many $\left \{ b_i \right \} \in \mathcal{Z}_1 \cap K$ satisfying \eqref{lowerbound}.
By passing to a subsequence if necessary, we can assume without loss of generality that $b_i \to b_{\infty} = 0 \in \mathcal{Z}_1 \cap K$.
Likewise, recalling \eqref{uppersemicontinuity} and passing to a subsequence if necessary, we can also assume that
\begin{equation*}
N_u \left ( b_i ; 0^+ \right ) \leq N_u \left ( b_{\infty} ; 0^+ \right ) + 1.
\end{equation*}
Using \eqref{subharmonicaverage}, \eqref{almgrenmonotonicity} and Lemma \ref{Hdimension}, it is easy to verify that the right hand side is finite, unless $u \equiv 0$ in $\Omega$.

Now that we have an upper-bound for respective frequencies at $\left \{ b_i \right \}$,
arguing exactly as in Corollary \ref{discretezeros}, for $\delta_0$ in \eqref{lowerbound} we can find positive $\left \{ r_i \right \}$, $\beta = \beta \left ( \delta_0 \right )$ 
and $\gamma = \gamma \left (\delta_0 \right )$ such that by Lemma \ref{3DHomogeneous} and \eqref{nonisolatedE2a}:
\begin{equation*}
N_u \left ( b_i ; r_i \gamma \right ) < \frac{1}{2 \sqrt{k}} + \delta_0,
\end{equation*}
for $i$ large enough, which together with \eqref{almgrenmonotonicity} contradicts the assumption \eqref{lowerbound}.
\end{proof}

\section{Structure of defects} \label{Structure}
 
In this section we study the structure of defects in the interior of a domain.  We begin with a strengthening of Lemma \ref{nonisolated1}. 

\begin{lemma} \label{StructureLemma}
For every $\epsilon > 0$, there exists a $\delta_0 = \delta_0 ( \epsilon )$ such that if $u \, : \, B_1(0) \to \mathbf{D}_k$ is an energy minimizing map satisfying:
\begin{equation*}
\left ( \overline{B_1(0)} \backslash B_{1/2}(0) \right ) \cap u^{-1} \{0\} \neq \emptyset, \quad \mathrm{and}
\end{equation*}
\begin{equation*}
N(0;2) < \frac{1}{2 \sqrt{k}} + \delta_0,
\end{equation*}
then for each $b \in B_1(0) \cap \mathcal{Z}_1$, $0<r \leq 1/2$, there exists $L_r^b$, a line passing through $b$, such that the following hold:
\begin{equation}
\overline{B_r(b)} \cap u^{-1} \{ 0\} \subset \left \{ x \, : \, \mathrm{dist} \left ( x, \mathcal{Z}_1 \right ) < r \epsilon \right \}, \quad \mathrm{and}
\label{StructureInclusion1}
\end{equation}
\begin{equation}
\overline{B_r(b)} \cap L_r^b \subset \left \{ x \, : \, \mathrm{dist} \left ( x, \mathcal{Z}_1 \right ) < r \epsilon \right \}.
\label{StructureInclusion2}
\end{equation}
Furthermore, there holds
\begin{equation}
B_{1/2}(0) \cap \mathcal{Z}_1 \subset \Gamma \subset B_1(0) \cap \mathcal{Z}_1,
\label{StructureInclusion3}
\end{equation}
for a single embedded H\"older continuous arc $\Gamma$. 
\end{lemma} 

\begin{proof}
Recall that for $v = w \circ p$ for $w$ as in \eqref{2Dhminimizer} and $p$ a projection, given $\epsilon > 0$,
arguing as in Lemma \ref{3DHomogeneous} there exists a $\beta_0 = \beta_0 \left ( \epsilon, k \right )$ such that
whenever $\mathrm{dist} \left ( b, p^{-1} \{ 0 \} \right ) < \beta_0$, there holds:
\begin{equation*}
N_v (b; 1) < \frac{1}{2 \sqrt{k}} + \frac{1}{2} \delta ( \epsilon).
\end{equation*}
Note that this estimate is stronger than \eqref{frequencyperturbation}, which holds for general homogeneneous minimizers,
and it is valid because $v$ is a cylindrical map.
Then using \eqref{nonisolated1Ea} and the trace theorem, we note that for all $b \in B_1(0) \cap \mathcal{Z}_1$,
\begin{equation*}
N_u (b; 1 ) < \frac{1}{2 \sqrt{k}} + \delta_0,
\end{equation*}
where $\delta_0 = \min \left \{ \delta(\epsilon) , \delta \left ( \beta_0 (\epsilon ) \right ) \right \}$.
Consequently, by the monotonicity \eqref{almgrenmonotonicity}: 
\begin{equation*}
\frac{1}{2 \sqrt{k}} \leq N_u \left ( b; 0^+ \right ) \leq N_u \left (b; 2r \right ) \leq N_u \left (b; 1 \right ) < \frac{1}{2 \sqrt{k} } + \delta(\epsilon),
\end{equation*}
for any $r \in (0, 1/2 ]$.

We would like to apply Lemma \ref{nonisolated1} to the map
\begin{equation*}
u_{b,r} = u \left ( b + r (.) \right ),
\end{equation*}
in order to conclude that:
\begin{equation*}
\overline{B_1(0)} \cap u_{b,r}^{-1} \{ 0 \} \subset  \left \{ x \, : \, \mathrm{dist} \left (x, p_{b,r}^{-1} \{ 0 \} \right ) < \epsilon \right \}, \; \mathrm{and}
\end{equation*}
\begin{equation*}
\overline{B_1(0)} \cap p_{b,r}^{-1} \{ 0 \} \subset \left \{ x \, : \, \mathrm{dist} \left (x, \mathcal{Z}_1 \right ) < \epsilon \right \},
\end{equation*}
which yields \eqref{StructureInclusion1} and \eqref{StructureInclusion2} with $L_{b,r} =  p_{b,r}^{-1} \{ 0 \}  + b$ for any $r \in (0, 1/2 ]$. 
While from above we have:
\begin{equation*}
N_{u_{b,r}} (0; 2) = N_u \left (b; 2r \right ) < \frac{1}{2 \sqrt{k} } + \delta(\epsilon),
\end{equation*}
as in the second part of the hypothesis of Lemma \ref{nonisolated1}, there remains to verify the first part of the hypothesis, which is:
\begin{equation*}
\left ( \overline{B_1 (0)} \backslash B_{1/2} (0) \right ) \cap u_{b,r}^{-1} \{0\} \neq \emptyset,
\end{equation*}
or equivalently:
\begin{equation}
\left ( \overline{B_r (b)} \backslash B_{r/2} (b) \right ) \cap u^{-1} \{0\} \neq \emptyset, 
\label{nofineholes}
\end{equation}
for all $r \in (0, 1/2 ]$. 

By Lemma \ref{nonisolated1} applied at $0$, there exists a line $L$ through $0$ such that:
\begin{equation*}
\begin{aligned}
& \overline{B_1(0)} \cap u^{-1} \{ 0 \} \subset \left \{ x \, : \, \mathrm{dist} \left ( x, L \right ) < \epsilon \right \} \quad \mathrm{and} \\
& \overline{B_1(0)} \cap L \subset \left \{ x \, : \, \mathrm{dist} \left ( x, \mathcal{Z}_1 \right ) < \epsilon \right \}.
\end{aligned}
\end{equation*}
Hence, the first inclusion implies that $b \in \mathcal{Z}_1$ lies in an $\epsilon$-neighborhood of $L$. Therefore, $\overline{B_{1/2}(b)} \backslash B_{1/4}(b)$ intersects
$L$ for $\epsilon$ small, whenever $\delta_0$ is sufficiently small. Moreover, for $\epsilon$ small enough, a ball of radius $2 \epsilon$ centered at a point in $L$
fits in $\overline{B_{1/2}(b)} \backslash B_{1/4}(b)$. Thus, the second inclusion implies \eqref{nofineholes} for $r = 1/2$, when $\epsilon$ is small enough, and as a result, the rescaled
version of Lemma \ref{nonisolated1} applies at $b$ with $r = 1/2$. Applying it gives \eqref{nofineholes} for $2 \epsilon$, and consequently the rescaled
version of Lemma \ref{nonisolated1} applies at $b$ with $r = 2 \epsilon$. 
By iteration with $r = 2 \epsilon, 4 \epsilon^2, 8 \epsilon^3...$
we conclude that \eqref{nofineholes} holds for all $r \in (0, 1/2 ]$. Therefore, Lemmas \ref{nonisolated1} and \ref{nonisolated2} apply to $u_{b,r}$ for all
$r \in (0, 1/2 ]$.

Finally, for $\epsilon$ small enough, the rest of the claim follows from Reifenberg's topological disk theorem, cf. \cite{Reifenberg60}, \cite{Morrey66}.
\end{proof}

Now we are ready to describe the zero set of energy minimizing maps in a neighborhood of any given zero.

\begin{theorem} \label{localpicture}
Suppose $u \, : \, \Omega \to \mathbf{D}_k$ is an energy minimizing map. Then each point $b \in u^{-1} \{ 0 \}$ has an open neighborhood $\mathcal{O}$ so that
$ \left ( \mathcal{O} \backslash \{ b \} \right ) \cap u^{-1} \{ 0 \}$ consists of 
an even number of disjoint, embedded H\"older continuous arcs $\Gamma_1, \Gamma_2, ..., \Gamma_{2m}$,
joining $\{b\}$ to a point in $\partial \mathcal{O}$. Moreover, $2m \leq K_0 \left ( N_u \left ( b; 0^+ \right ) \right )$, 
and for $1 \leq i < j \leq 2m$, $r$ sufficiently small and $d_0 = d_0 \left ( N_u \left ( b; 0^+ \right ) \right )$,
$\mathrm{dist} \left ( \Gamma_i \cap \partial B_r(a) , \Gamma_j \cap \partial B_r (a) \right ) \geq \frac{1}{2} d_0 r$.
\end{theorem}

\begin{proof}
In case $b \in \mathcal{Z}_0$, $b$ is an isolated singularity, and consequently the claim holds trivially. Therefore, we assume $b \in \mathcal{Z}_1$.
By Corollary \ref{discretezeros}, after a suitable translation and scaling, we may also assume that $b = 0$, $\Omega = B_1(0)$ and $\mathcal{Z}_0 = \emptyset$.

For $\epsilon > 0$ to be determined, we would like to apply Lemma \ref{nonisolated2} to $u$ in order to obtain an $R = R \left ( \epsilon, N_u \left ( 0; 0^+ \right ) \right ) > 0$ and
an approximating homogeneous minimizer $v_r$ for each $r \in (0, R]$. 
By \eqref{nonisolatedE2c} there will be at least a point $b_{a,r} \in \partial B_{\frac{3r}{4}}(0) \cap B_{r \epsilon} \left ( \frac{3a}{4} \right ) \cap u^{-1} \{ 0 \}$
for each $a \in v_r^{-1} \{ 0 \} \cap \mathbf{S}^2$.
Before applying Lemma \ref{nonisolated2}, first we choose $\delta = \delta \left ( d_0 /16 \right )$ as in Lemma \ref{StructureLemma}, 
where $d_0$ is as in Lemma \ref{3DHomogeneous}, depending on $N_u \left (0; 0^+ \right )$ only.
Secondly, we choose $\beta = \beta \left ( d_0/2, N_u \left ( 0; 0^+ \right )  \right )$ and $\gamma = \gamma \left ( d_0/2, N_u \left ( 0; 0^+ \right )  \right )$ 
as in Lemma \ref{3DHomogeneous}. 
Finally, we choose $R = R \left ( \epsilon, N_u \left ( 0; 0^+ \right ) \right )$ for $\epsilon$ yet to be determined. Hence, we can update all our parameters in chain, depending
on our choice of $\epsilon$.

For such parameters we obtain:
\begin{equation*}
N_u \left ( b_{a,r} ; 2r \gamma \right ) \leq N_{v_r} \left ( b_{a,r}; 2r \gamma \right ) + \delta/2 < \frac{1}{2 \sqrt{k}} + \delta,
\end{equation*}
where the first inequality follows from \eqref{nonisolatedE2a} for $\epsilon$ chosen sufficiently small, and the second inequality follows from Lemma \ref{3DHomogeneous},
as $\beta$ and $\gamma$ have been chosen sufficiently small with respect to $\delta$. Shrinking $\epsilon$ further so that it is much smaller than $\gamma/2$,
we claim $ \left ( \overline{B_{r\gamma} \left ( b_{a,r} \right )} \backslash B_{\frac{r\gamma}{2}} \left ( b_{a,r} \right ) \right ) \cap u^{-1} \{ 0 \} \neq \emptyset$. 
This follows from the inclusion \eqref{nonisolatedE2c} and
the observation that $B_{r \epsilon} (p)$ is contained the annulus 
$ \overline{B_{r\gamma} \left ( b_{a,r} \right )} \backslash B_{\frac{r\gamma}{2}} \left ( b_{a,r} \right ) $ 
for some $p \in v_r^{-1} \{0\}$, as $ \left | b_{a,r} - \frac{3}{4}a \right | < r \epsilon$. 
Thus, the mapping $u_{a,r} (x) = u \left ( b_{a,r} + r \gamma x \right )$ satisfies the hypothesis of Lemma \ref{StructureLemma}. 
Applying it, by the choice $\delta = \delta \left ( d_0 /16 \right )$, we conclude:
\begin{equation*}
\overline{B_{\frac{r \gamma}{2}} \left ( b_{a,r} \right )} \cap u^{-1} \{0\} \subset \Gamma_{a,r} \subset B_{r\gamma} \left ( b_{a,r} \right ) \cap u^{-1} \{ 0 \}
\subset \left \{ x \, : \, \mathrm{dist} \left ( x, L_{a,r} \right ) < d_0 r \gamma /16 \right \},
\end{equation*}
for some embedded H\"older continuous arc $\Gamma_{a,r}$ and some line $L_{a,r}$ passing through $b_{a,r}$. 

Once again recalling that $ \left | b_{a,r} - \frac{3}{4}a \right | < r \epsilon$, where $\epsilon$ has been chosen to be much smaller than $\gamma/2$, we obtain:
\begin{equation*}
\left ( \overline{B_{\frac{3}{4}r + \frac{1}{4}r\gamma }(0)} \backslash B_{\frac{3}{4}r - \frac{1}{4}r\gamma }(0) \right ) 
\cap \left \{ x \, : \, \mathrm{dist} \left ( x, v_r^{-1} \{ 0 \} \right ) < \epsilon r \right \} \subset \overline{B_{\frac{r\gamma}{2}} \left ( b_{a,r} \right )},
\end{equation*}
for each $a \in v_r^{-1} \{ 0 \} \cap \mathbf{S}^2$ and the corresponding $b_{a,r}$. Consequently:
\begin{equation*}
\left ( \overline{B_s(0)} \backslash B_{\lambda s}(0) \right ) \cap u^{-1} \{ 0 \} \subset \bigcup_{a \in v_r^{-1} \{ 0 \} } \Gamma_{a,r} \subset u^{-1} \{ 0 \},
\end{equation*}
where $s = \frac{3}{4}r + \frac{1}{4} r \gamma $ and $ \lambda = 1 - \frac{2 \gamma}{3+\gamma}$.
By the inclusion \eqref{nonisolatedE2c}, each arc $\Gamma_{a,r}$ intersects both the outer sphere $\partial B_s(0)$ and the inner sphere $\partial B_{\gamma s}(0)$.
Arguing as in Lemma \ref{nonisolated2} we infer that when $r$ is sufficiently small, $\Gamma_{a,r}$ overlaps with $\Gamma_{\bar{a}, \lambda r }$, when $\bar{a}$ is the nearest
point to $a$ in $v_{\lambda r}^{-1} \{ 0 \} \cap \mathbf{S}^2$. We finally note that $\Gamma_{a,r} \cup \Gamma_{\bar{a}, \lambda r}$ is clearly also a H\"older continuous arc.

Now starting with $r = R$, we iterate the above argument with $r = R, \lambda R, \lambda^2 R, ...$ Thus, we obtain the arcs $\Gamma_1, \Gamma_2, ..., \Gamma_{2m}$,
by forming for each $a \in v_R^{-1} \{ 0 \} \cap \mathbf{S}^2$, a union of chains of overlapping arcs starting with $\Gamma_{a,R}$. We let $\mathcal{O}$ to be $B_{\frac{3+\gamma}{4}R}(0)$ for
$R = R \left ( \epsilon, N \left ( 0; 0^+ \right ) \right )$.
Lastly, the bound on $2m$ and the minimal distance
between any two arcs at scale $r$ both follow from the corresponding conclusions of Lemma \ref{3DHomogeneous}.
\end{proof}

The main result follows immediately from Theorem \ref{localpicture}:

\begin{corollary} \label{mainresult}
For any energy minimizing map $u \, : \, \Omega \to \mathbf{D}_k$ and any compact $K \subset \Omega$, 
$u^{-1} \{ 0 \} \cap K$ consists of isolated points and finitely many H\"older curves with only finitely many crossings.
\end{corollary}

\section*{Acknowledgments}
The first and third authors were in part supported by the National Science Foundation grant DMS-1501000. 
The second author was in part supported by the National Science Foundation grant DMS-1207702.

\bibliography{bibext}
\bibliographystyle{amsplain}
\end{document}